\documentclass[12pt,reqno]{amsart}

\usepackage{amsmath,amsfonts,amsthm,graphicx,float,color}
\usepackage{cite}
\usepackage[hidelinks]{hyperref}
\usepackage{xcolor}

\usepackage{cancel}
\usepackage[margin = 1in]{geometry}
\usepackage{mathrsfs}
\usepackage{mathabx}
\usepackage{dsfont}

\renewcommand{\d}{\delta}
\newcommand{\N}{\mathbb{N}}

\newcommand{\R}{\mathbb{R}}
\newcommand{\Z}{\mathbb{Z}}
\newcommand{\C}{\mathbb{C}}

\newcommand{\ve}{\varepsilon}

\newcommand{\lp}{\left(}
\newcommand{\rp}{\right)}

\newcommand{\wt}{\widetilde}
\newcommand{\wh}{\widehat}

\newcommand{\p}{\partial}

\newcommand{\Op}{\textnormal{Op}}
\newcommand{\nm}[1]{\left| \left| #1 \right| \right|}

\newcommand{\untab}{\hspace{-.2in}}
\newcommand{\<}{\langle}
\renewcommand{\>}{\rangle}
\newcommand{\supp}{\textnormal{supp\,}}
\newcommand{\ltwo}[1]{\lpn{#1}{2}}
\newcommand{\ltwoo}[2]{\nm{#1}_{L^2(#2)}}
\newcommand{\lpn}[2]{ \nm{#1}_{L^{#2}}}

\newcommand{\vphi}{\varphi}
\newcommand{\Rn}{\R^n}

\newcommand{\spacecutoff}{\chi}
\newcommand{\freqcutoff}{\psi}

\newcommand{\II}{I\!\!I}

\newcommand{\ra}{\rightarrow}

\newcommand{\charset}{\textnormal{Char}}

\DeclareMathOperator{\Id}{Id}
\DeclareMathOperator{\Spec}{Spec}

\renewcommand{\d}{\delta}
\newcommand{\Rb}{\mathbb{R}}

\newcommand{\Real}{\textnormal{Re\,}}

\newcommand{\Dc}{\mathcal{D}}
\newcommand{\Ec}{\mathcal{E}}

\newcommand{\Cb}{\mathbb{C}}
\newcommand{\intrn}{\int_{\Rn}}

\makeatletter
\def\thmhead@plain#1#2#3{%
  \thmname{#1}\thmnumber{\@ifnotempty{#1}{ }\@upn{#2}}%
  \thmnote{ {\the\thm@notefont#3}}}
\let\thmhead\thmhead@plain
\makeatother

\newtheorem{theorem}{Theorem}
\newtheorem{lemma}{Lemma}
\newtheorem{prop}[lemma]{Proposition}

\newtheorem{rmk}[lemma]{Remark}
\newtheorem{example}[lemma]{Example}
\newtheorem{definition}[lemma]{Definition}
\newtheorem{assumption}{Assumption}

\newcommand{\thmref}[1]{Theorem~\ref{#1}}
\newcommand{\thmsref}[2]{Theorems~\ref{#1} and \ref{#2}}
\newcommand{\secref}[1]{Section~\ref{#1}}
\newcommand{\secsref}[2]{Sections~\ref{#1} and \ref{#2}}
\newcommand{\lemref}[1]{Lemma~\ref{#1}}

\newcommand{\propref}[1]{Proposition~\ref{#1}}

\newcommand{\assumref}[1]{Assumption~\ref{#1}}
\newcommand{\assumsref}[2]{Assumptions~\ref{#1} and \ref{#2}}

\definecolor{bpurple}{RGB}{170,0,200}

\numberwithin{equation}{section}
\numberwithin{lemma}{section}

\definecolor{bpurple}{RGB}{170,0,200}

\begin{document}

\title[Sharp decay for anisotropically damped waves]{Sharp exponential decay rates for anisotropically damped waves}

\author[B. Keeler]{Blake Keeler}
\email{\href{mailto:bkeeler@live.unc.edu}{bkeeler@live.unc.edu}}
\address{Department of Mathematics and Statistics \\ McGill University \\ 805 Rue Sherbrooke Ouest \\ Montréal, QC H3A 0B9}

\author[P. Kleinhenz]{Perry Kleinhenz}
\email{\href{mailto:pkleinhe@gmail.com}{pkleinhe@gmail.com}}
\address{Department of Mathematics, Michigan State University \\ 619 Cedar River Rd \\ East Lansing, MI 48823}

\date{\today}
\maketitle

\begin{abstract}
In this article, we study energy decay of the damped wave equation on compact Riemannian manifolds where the damping coefficient is anisotropic and modeled by a pseudodifferential operator of order zero. We prove that the energy of solutions decays at an exponential rate if and only if the damping coefficient satisfies an anisotropic analogue of the classical geometric control condition, along with a unique continuation hypothesis. Furthermore, we compute an explicit formula for the optimal decay rate in terms of the spectral abscissa and the long-time averages of the principal symbol of the damping over geodesics, in analogy to the work of Lebeau for the isotropic case. We also construct genuinely anisotropic dampings which satisfy our hypotheses on the flat torus. 
\end{abstract}

\section{introduction}\label{Introduction}
Let $(M,g)$ be a smooth, compact manifold without boundary and let $\Delta_g$ be the associated Laplace-Beltrami operator (taken with the convention that $\Delta_g \le 0$). Suppose ${W: L^2(M) \ra L^2(M)}$ is bounded and nonnegative. We consider the generalized damped wave equation given by
\begin{equation}\label{damped_wave}
\begin{cases}
\partial_t^2u - \Delta_g u + 2W\partial_t u = 0 \\
(u, \partial_t u)|_{t=0} = (u_0, u_1),
\end{cases}
\end{equation}
for $(u_0,u_1)^T\in \mathscr H:= H^1(M) \oplus L^2(M)$, where $\mathscr H$ is taken with the natural norm
\[\|(u_0,u_1)^T\|_{\mathscr H}^2 = \|(1-\Delta_g)^{\frac{1}{2}}u_0\|_{L^2(M)}^2 + \|u_1\|_{L^2(M)}^2.\]
\noindent We study the asymptotic properties of the energy of solutions to $\eqref{damped_wave}$ as $t \ra \infty$. Here, the energy is defined by 
\begin{equation}\label{energy}
E(u,t) =  \frac{1}{2}\int_M |\nabla_g u(t,x)|^2 + |\partial_t u(t,x)|^2 \,dv_g(x),
\end{equation}
where $dv_g$ is the Riemannian volume form on $M.$ It is straightforward to compute that
\begin{equation}\label{dissipation}
\frac{d}{dt}E(u,t) = - 2\Real\langle W\partial_t u,\partial_t u\rangle \leq 0,
\end{equation}
where $\langle\cdot,\cdot\rangle$ denotes the inner product on $L^2(M,g).$ Thus, the assumption that $W$ is a nonnegative operator guarantees that the energy of solutions to \eqref{damped_wave} experiences dissipation, but \eqref{dissipation} does not indicate how quickly the energy decays as $t\to\infty$. The most straightforward type of decay is uniform stabilization, i.e. when there exists a constant $C>0$ and a real-valued function $t\mapsto r(t)$ with $r(t)\to 0$ as $t \ra \infty$ such that 
$$
E(u,t) \leq C r(t) E(u,0).
$$
In the case where $W$ acts via multiplication by a bounded, nonnegative function $b$, a great deal is known about energy decay rates. Perhaps the most well known result states that solutions to \eqref{damped_wave} experience uniform stabilization with an exponential rate if and only if $W$ satisfies the geometric control condition (GCC) \cite{RauchTaylor1975,Ralston1969}. The GCC is satisfied if there exists some $T>0$ such that every geodesic with length at least $T$ intersects the set where $b$ is bounded below by some positive constant. Many other works have proved weaker decay rates in the setting where the GCC is not satisfied (c.f. \cite{Lebeau1996} \cite{Burq1998}  \cite{BurqZuily2016}
\cite{Christianson2007} \cite{Christianson2010} \cite{BurqChristianson2015} 
\cite{LiuRao2005} \cite{BurqHitrik2007}). With more restrictive assumptions on $W$ and $M$, one can obtain sharp decay rates (c.f.\cite{AnantharamanLeautaud2014} \cite{Stahn2017} \cite{LeautaudLerner2017} \cite{Kleinhenz2019} \cite{DatchevKleinhenz2020} \cite{Kleinhenz2019a}, \cite{Sun2022}
\cite{DyatlovJinNonnenmacher} \cite{Jin2020}). 

A distinct shortcoming of the multiplicative case is that the damping force is sensitive only to positional information and not to the direction in which the solution propagates. For this reason, one can classify multiplicative damping as an isotropic force, but many physical systems which experience \emph{anisotropic} damping forces are studied in materials science, physics, and engineering \cite{Krattiger2016, Craig2008, Joubert2011}. However, a general analysis of the damped wave equation in the anisotropic case has not yet been done. This article aims to address this gap in the literature by studying the case where the anisotropic damping force is modeled by a pseudodifferential operator.

It is common in analysis of the generalized damped wave equation \eqref{damped_wave} to assume that $W$ takes the form of a square, i.e. $W = B^*B$ for some bounded operator $B$ (c.f. \cite{AnantharamanLeautaud2014}). This guarantees that $W$ is nonnegative and enables the use of certain techniques from spectral theory. We allow for a slightly more general assumption here, namely that $W$ takes the form $W = \sum\limits_{j=1}^N B_j^*B_j$ for some finite collection $\{B_j\}_{j=1}^N\subset\Psi_{c\ell}^0(M)$, where $\Psi_{cl}^0(M)$ denotes the space of classical pseudodifferential operators 
on $M$ of order zero with polyhomogeneous symbols. The corresponding space of symbols is denoted $S_{c\ell}^0(T^*M)$. We note that allowing $W$ to take the form of a sum of squares is indeed a generalization, since it is not generically possible to write $\sum_{j=1}^N B_j^*B_j$ as $B^*B$ for some $B\in \Psi_{c\ell}^0(M)$, since the pseudodifferential calculus only allows for the computation of square roots modulo a smoothing remainder. We denote by $w\in S_{c\ell}^0(T^*M)$ the principal symbol of $W$, taken to be positively fiber-homogeneous of degree 0 outside a small neighborhood of the zero section in $T^*M.$ That is, $w(x,s\xi) = w(x,\xi)$ for all $s > 0$ and all $|\xi| \ge c$ for some $c > 0$ which can be chosen to be arbitrarily small. This homogeneity allows us to treat $w$ as a function on the co-sphere bundle 
\[S^*M := \{(x,\xi)\in T^*M:\, |\xi|_g = \frac{1}{2}\},\]
where the choice of $\frac{1}{2}$ is made for the sake of convenience in later arguments. 

We now state the required assumptions for the main theorem. The first is an anisotropic analogue of the classical geometric control condition.

\begin{assumption}[(\textbf{Anisotropic Geometric Control Condition})]\label{geometric_control}
Let $\varphi_t$ denote the lift of the geodesic flow to $T^*M$. Assume that there exists a compact neighborhood $K$ of the zero section in $T^*M$ and constants $T_0,c>0$ such that for every $(x_0,\xi_0)\in T^*M\setminus K$,
\[\frac{1}{T}\int\limits_0^T w(\varphi_t(x_0,\xi_0))\,dt \ge c,\quad \text{for }T \ge T_0.\]
That is, the long-time averages of $w$ over geodesics are uniformly bounded below. In this case, we say $W$ satisfies the \emph{anisotropic} geometric control condition (AGCC).
\end{assumption}
\noindent Note that in the case of multiplicative damping, \assumref{geometric_control} is equivalent to classical geometric control condition in \cite{RauchTaylor1975}. 

The second key assumption requires that the kernel of $W$ contain no nontrivial eigenfunctions of $\Delta_g$. 
\begin{assumption}\label{unique_cont}
 If $v\in L^2(M)$ satisfies $-\Delta_g v = \lambda^2 v$ with $\lambda \neq 0$, then $W v \ne 0.$ 
\end{assumption} 
\noindent In the case where $W = b(x)$, \assumref{unique_cont} is satisfied when $b$ is supported on any open set, since eigenfunctions of $\Delta_g$ cannot vanish on open sets by the unique continuation principle (c.f. \cite{RauchTaylor1975}). It is for this reason that we sometimes refer to \assumref{unique_cont} as a ``unique continuation hypothesis." 

With these assumptions stated, we then have the following equivalence.
\begin{theorem}\label{exp_decay_thm}
All solutions $u$ to \eqref{damped_wave} with $W\in\Psi_{c\ell}^0(M)$ satisfy 
\begin{equation}\label{exp_decay_eqn}
E(u,t)\le C e^{-\beta t}E(u,0)
\end{equation}
for some $C,\,\beta> 0$ and for all $t\ge 0$ if and only if $W$ satisfies \assumsref{geometric_control}{unique_cont}.
\end{theorem}

\noindent In other words, solutions experience uniform stabilization at an exponential rate if and only if $W$ satisfies  \assumsref{geometric_control}{unique_cont}.

The existing literature on anisotropic damping coefficients is quite limited. In the context of pseudodifferential $W$, Sj\"ostrand \cite{Sjostrand2000} studied the asymptotic distribution of eigenvalues of the stationary damped wave equation. Christianson, Schenck, Vasy, and Wunsch \cite{csvw} showed that a polynomial resolvent estimate for a related complex absorbing potential problem gives another polynomial resolvent estimate of the same order for the stationary damped wave equation. However, these results do not consider anisotropic damping in a time-dependent setting and so do not provide energy decay results. \thmref{exp_decay_thm} addresses this gap in the literature by providing conditions which guarantee exponential uniform stabilization, in analogy to the classical result of Rauch and Taylor \cite{RauchTaylor1975}.

Since \thmref{exp_decay_thm} only claims the existence of some exponential decay rate $\beta$, a natural question is to determine the optimal rate of decay for a given damping coefficient. Given a fixed $W\in\Psi_{c\ell}^{0}(M)$, we define the \emph{best exponential decay rate} as in \cite{Lebeau1996} via
\begin{equation}\label{alpha_def}
\alpha := \sup\{\beta\in\R:\, \exists C > 0 \text{ such that } E(u,t)\le Ce^{-\beta t}E(u,0) \,\,\,\forall u\text{ which solve \eqref{damped_wave}}\}.
\end{equation}
Our next result shows that $\alpha$ can be expressed in terms two fundamental quantities: the spectral abscissa, and the long-time averages of the principal symbol of $W$ over geodesics. The spectral abscissa is defined with respect to
\[A_W := \lp\begin{array}{cc}0 & \Id \\ \Delta_g & -2W\end{array}\rp,\]
which is the infinitesimal generator of the solution semigroup for \eqref{damped_wave}. For each $R > 0$, we set 
\[D(R) =  \sup\{\Real(\lambda):\,|\lambda| > R,\, \lambda\in \textnormal{Spec}(A_W)\}.\]
We then define the spectral abscissa as 
\begin{equation}\label{spectral_gap}
D_0 = \lim\limits_{R\to 0^+}D(R).
\end{equation}
We also define for $t \in \R$ the time-average of the damping along geodesics 
\[L(t) = \inf\limits_{(x,\xi)\in S^*M}\frac{1}{t}\int\limits_0^t w(\varphi_s(x,\xi))\,ds,\]
and the long-time limit
\begin{equation}\label{time_avg}
L_\infty = \lim\limits_{t\to\infty} L(t).
\end{equation}

\noindent We can then characterize $\alpha$ as follows. 

\begin{theorem}\label{best_const_thm}
The best exponential decay rate for solutions to \eqref{damped_wave} with $W\in \Psi_{c\ell}^0(M)$ is 
\[\alpha = 2\min\{-D_0,L_\infty\},\]
where $D_0$ and $L_\infty$ are defined by \eqref{spectral_gap} and \eqref{time_avg}, respectively. 
\end{theorem}

\begin{rmk}
\textnormal{
It is of considerable note that the formula for the optimal decay rate here is an exact analogy of the multiplicative case studied by Lebeau (c.f. \cite[Theorem 2]{Lebeau1996}). While the broad structure of our proof is similar, there are portions of the analysis which diverge greatly, particularly in \secref{pseudocoherentsection} where we investigate the action of pseudodifferential operators on Gaussian beams.
}
\end{rmk}

\begin{rmk}
\textnormal{
\thmref{best_const_thm} is significantly stronger than \thmref{exp_decay_thm}, although this is not immediately obvious. The main portion of this article is dedicated to the proof of \thmref{best_const_thm}. We then show that \thmref{best_const_thm} implies \thmref{exp_decay_thm} in \secref{thm_1_proof}.
}
\end{rmk}

\thmsref{exp_decay_thm}{best_const_thm} fit into a broad range of existing results which attempt to reproduce the equivalence of the GCC and exponential decay under modified hypotheses on the damping. It is not uncommon for such statements to be somewhat inconclusive. For example, when the damping is allowed to be time-dependent \cite{RousseaLebeauPeppinoTrelat} showed that for time periodic damping, the GCC is indeed equivalent to exponential decay, but it is not currently known if this result is true for non-periodic damping. In the setting where the damping is allowed to take negative values (commonly called ``indefinite damping"), the state of the art is similarly mixed. If $M$ is an open domain in $\Rn$ with $C^2$ boundary, \cite{LiuRaoZhang2002} proves an exponential decay rate provided that the damping is positive in a neighborhood of $\p M$ (which implies the GCC) and $\inf_{x\in M}W(x)$ is not too negative.  However, it is currently not known if an appropriate generalization of the GCC is equivalent to exponential stability in the indefinite case. The limitations of these results illustrate that seemingly simple changes to hypotheses on the damping coefficient can create substantial barriers to reproducing the classical equivalence theorem. So, the fact that \thmref{exp_decay_thm} provides a direct analogy of the GCC for pseudodifferential damping which is equivalent to exponential decay is somewhat exceptional. Note also that these other generalizations do not possess an analogy of \thmref{best_const_thm}. Although \cite{LiuRaoZhang2002} and \cite{RousseaLebeauPeppinoTrelat} both provide a rate for the exponential decay, it is not shown to be sharp.

Our final result concerns \assumref{unique_cont}, which is necessary in order to obtain \thmref{exp_decay_thm}. To see this, suppose that $v$ satisfies $-\Delta_g v = \lambda^2 v$ with $\lambda \ne 0$ and $Wv = 0$. Then, the function
\[u(t,x) = e^{it\lambda}v(x),\]
solves \eqref{damped_wave}, but has energy $E(u,t) = \lambda^2\|v\|^2_{L^2(M)}$ for all $t$. As previously mentioned, when $W$ is a multiplication operator supported on any open set, unique continuation results guarantee that $W$ does not annihilate any eigenfunctions of $\Delta_g$, making \assumref{unique_cont} unnecessary. However, in the pseudodifferential setting, verifying this assumption is more difficult.

A special case in which \assumref{unique_cont} is easy to check is when $W$ is constructed from functions of $\Delta_g$. Suppose $W = B^*B$ with $B = f(-\Delta_g)$, where $f:\R\to\R$ satisfies a ``symbol-type" estimate of the form 
\[|\partial_s^k f| \le C (1 + |s|)^{-k}\]
for any $k.$ The functional calculus of Strichartz \cite{Strichartz1972} shows that $W$ is pseudodifferential when constructed in this way. The calculus also immediately implies that \assumref{unique_cont} holds as long as $f$ does not vanish on the spectrum of $-\Delta_g$, since for any eigenfunction $v$ with eigenvalue $\lambda$, we have $Wv = f(\lambda)^2 v$. However, damping coefficients constructed in this fashion are somewhat uninteresting in the sense that the principal symbol is a function of $|\xi|_g^2$, and therefore independent of direction. Thus, examples of this type are not truly anisotropic. In general, it is not obvious that one can always construct nontrivial anisotropic examples satisfying \assumref{unique_cont}, although we expect that a rich class of examples do indeed exist. The following theorem demonstrates that one can always produce such examples when $(M,g)$ is real analytic. 

\begin{theorem}\label{unique_cont_thm}
If $(M,g)$ is compact and real analytic, then there exists $W \in \Psi_{cl}^0(M)$ of the form $W=\sum_{j=1}^N B_j^*B_j$,such that for each $x\in M$, the principal symbol of $W$ vanishes on an open cone in $T_x^*M$ and for any $v$ an eigenfunction of $\Delta_g, Wv \ne 0.$ 
\end{theorem}
\noindent The fact that the principal symbol vanishes in an open cone of directions at each point implies that the $W$ in this theorem is not built from functions of $\Delta_g$, excluding the somewhat trivial case discussed previously. Using the machinery developed in the proof of \thmref{unique_cont_thm}, we are able to produce explicit examples on the flat 2-torus of operators $W\in\Psi_{cl}^0$ which satisfy both \assumsref{geometric_control}{unique_cont}. This construction is presented in Section \ref{examples}.

 \begin{rmk}
 \textnormal{As mentioned previously, \assumref{unique_cont} follows directly from the geometric control condition in the multiplicative case. One might hope that this could be generalized to the scenario where $W$ is pseudodifferential, but this problem is exceedingly difficult in general. The only result of this type known to the authors is that of \cite{DyatlovJinNonnenmacher2019}, which utilized the fractal uncertainty principle to show that when $M$ is an Anosov surface, $v$ is an eigenfunction of the Laplacian, and the principal symbol of $W$ is not identically zero, then there is a quantitative lower bound on the size of $Wv$. So for Anosov surfaces, \assumref{geometric_control} implies \assumref{unique_cont}. But even in this specialized case, the proof involves highly sophisticated techniques. An analysis of the general case is an open problem, and we suspect that it would be a significant undertaking. 
}
\end{rmk}

\subsection{Outline of the article}
The majority of this article is devoted to the proof of \thmref{best_const_thm}, which spans Sections \ref{defectmeasure}, \ref{pseudocoherentsection} \ref{upper}, and \ref{lower}. The primary tool is microlocal defect measures. We begin in \secref{defectmeasure} by analyzing the behavior of defect measures associated to sequences of solutions to \eqref{damped_wave} when propagated by the Hamiltonian flow for $P = \partial_t^2 - \Delta_g$. The formula we produce follows largely from direct computations and measure-theoretic arguments, in analogy to \cite{Lebeau1996,Klein2017}. In \secref{pseudocoherentsection}, we perform a detailed study of the action of certain pseudodifferential operators on coherent states, which is a critical component of constructing quasimodes for \eqref{damped_wave}. This analysis is a key point where the pseudodifferential case becomes significantly more difficult than the multiplicative setting. In \secref{upper}, we then combine the results of \secsref{defectmeasure}{pseudocoherentsection} to produce quasimodes for the damped wave equation whose energy is strongly localized near a fixed geodesic. The analysis of these quasimodes allows us to prove that $\alpha \leq 2\min\{-D_0,L_\infty\}$. The proof of \thmref{best_const_thm} is completed in \secref{lower}, where we prove the lower bound $\alpha \ge 2\min\{-D_0,L_\infty\}.$ This section follows in close analogy to \cite{Lebeau1996}, and so we omit some of the more technical details. 

In \secref{thm_1_proof}, we show that \thmref{best_const_thm} implies \thmref{exp_decay_thm}. This follows directly from spectral theory analysis. 

Finally, in \secref{examples}, we restrict to the case of real analytic manifolds to produce some examples. We provide a fairly generic condition on pseudodifferential operators which guarantees that they satisfy \assumref{unique_cont}. We also show that one can always produce examples which fall into this category, thus proving \thmref{unique_cont_thm}. We then conclude by constructing some explicit examples on the flat torus which satisfy both \assumsref{geometric_control}{unique_cont}. 

\subsection{Acknowledgements}
The authors would like to thank J. Wunsch and A. Vasy for their frequent helpful comments throughout the course of this project. Wunsch's suggestions regarding the analytic wave front set argument in Section \ref{examples} were particularly useful. We are also grateful to M. Taylor for helping us better understand the proof of his original result with J. Rauch on exponential energy decay in [RT75]. We also wish to thank H. Christianson, Y. Canzani and J. Galkowski for their comments on an earlier version of this paper. Finally, BK was supported in part by NSF Grant DMS-1900519 through his advisor Y. Canzani.

\section{Propagation of the microlocal defect measure}\label{defectmeasure}

\noindent In this section, we compute the propagation of microlocal defect measures associated to a sequence of solutions of \eqref{damped_wave} with $W\in\Psi^0_{c\ell}(M)$. We begin by noting some general facts about microlocal defect measures. We will not prove these here, and we direct the reader to the seminal article of G\'{e}rard \cite{Gerard1991} for details and proofs. Here we consider defect measures as measures on $S^*(\R\times M) = \{(t,x,\tau,\xi)\in T^*(\R\times M):\, \tau^2 + |\xi|_g^2 = \frac{1}{2}\},$ the co-sphere bundle of $\R\times M$ treated as a Riemannian manifold with the the product metric. For any sequence $\{u_k\}\subseteq H^m(\R\times M)$ converging weakly to 0, there exists a sub-sequence $\{u_{k_j}\}$ and a positive Radon measure $\nu$ on $S^*(\R\times M)$ such that for any $A\in\Psi_{c\ell}^{2m}(\R\times M)$ with compact support in $t$, we have
\[\lim\limits_{j\to\infty}\langle Au_{k_j},u_{k_j}\rangle = \int\limits_{S^*(\R\times M)}\!\!\!\!a\,d\nu,\]
where $\langle\cdot,\cdot\rangle$ denotes the standard inner product on $L^2(\R\times M)$, which is the natural pairing between $H^{-1}(\R\times M)$ and $H^1(\R\times M).$ If $\{u_k\}$ has a defect measure without the need for passing to a subsequence, we say that $\{u_k\}$ is \emph{pure}.

From this point on, we specialize to the case where $\{u_k\}\subseteq H^1(\R\times M)$ is a pure sequence of solutions to the damped wave equation converging weakly to zero, with associated defect measure $\nu$. A key piece of the proof of \thmref{best_const_thm} is the propagation of this defect measure under the Hamiltonian flow on $T^*(\R\times M)$ generated by $p(t,x,\tau,\xi) = |\xi|^2_g -\tau^2$, the principal symbol of $P = \partial_t^2 - \Delta_g$. We denote this flow by $\Phi_s$, and it can be written as
\[\Phi_s(t,x,\tau,\xi) = (t-2s\tau,\tau,\varphi_s(x,\xi)),\]
where we recall that $\varphi_s$ is the geodesic flow on $T^*M$. More precisely, the propagation of the defect measure refers to the behavior of $\nu$ under the pushforward $(\Phi_s)_*$. 

We now show that there exists a smooth function $s\mapsto G_s \in C^\infty(S^*(\R\times M))$ such that $(\Phi_s)_*\nu = G_{-s}\nu$, and that $G_s$ can be defined as the solution to a certain differential equation. 

\begin{lemma}\label{defect_measure_propagation}
For any fixed $(t,x,\tau,\xi)\in S^*(\R\times M)$, define $G_s(t,x,\tau,\xi)$ as the solution to the initial value problem
\begin{equation}\label{G_s_defn}
\begin{cases}
G_0(t,x,\tau,\xi)  = 1,\\
\partial_s G_s(t,x,\tau,\xi) = \{p,G_s\}(t,x,\tau,\xi) + 4\tau w(x,\xi) G_s(t,x,\tau,\xi),
\end{cases}
\end{equation}
where $p$ and $w$ are the principal symbols of $P=\partial_t^2 - \Delta_g$ and $W$, respectively. Then, ${(\Phi_s)_*\nu = G_{_{\!\!-s}}\nu}$. Equivalently, for any $b\in C^\infty(S^*(\R\times M))$ which is compactly supported in $(t,x)$, 
\begin{equation}\label{propagation_eqn}
\int\limits_{S^*(\R\times M)}\untab b\circ \Phi_s\,d\nu \,\,\,=\!\!\! \int\limits_{S^*(\R\times M)}\untab b\,G_{_{\!\!-s}}\,d\nu.
\end{equation}
\end{lemma}

\begin{rmk}
\textnormal{
Note that the content of \lemref{defect_measure_propagation} is analogous to that of \cite[Proposition 8]{Klein2017}, and the proof goes through in a very similar fashion for the case of pseudodifferential damping. 
}
\end{rmk}

\begin{proof}
First, observe that in order to prove \eqref{propagation_eqn}, it is sufficient to show that for all $s \in \Rb$ and any $b\in C^{\infty}(S^*(\R\times M))$ with compact support in $(t,x)$
\[
\int \limits_{S^*(\R\times M)} (b \circ \Phi_s) G_{s} \, d\nu = \int  \limits_{S^*(\R\times M)} b\, d\nu. 
\]
Furthermore, since $G_0=1$, this is equivalent to showing that
\[\partial_s\untab\int\limits_{S^*(\R\times M)}\untab(b\circ \Phi_s)G_{s}\,d\nu = 0.\]
By direct computation, we see that 
\[\partial_s\left[(b\circ \Phi_s)G_s\right] = \{p,b\circ\Phi_s\}G_s + (b\circ \Phi_s)\partial_s G_s,\]
since $\Phi_s$ is the Hamiltonian flow generated by $p.$ Using algebraic properties of the Poisson bracket, we have $\{p,b\circ\Phi_s\}G_s = \{p,(b\circ \Phi_s)G_s\} - \{p,G_s\}(b\circ \Phi_s).$ Therefore,
\begin{equation}\label{propagation_product}
\partial_s\untab\int\limits_{S^*(\R\times M)}\untab(b\circ \Phi_s)G_s\,d\nu \,\,\, =\!\!\! \int\limits_{S^*(\R\times M)}\untab\{p,(b\circ\Phi_s)G_s\} - \{p,G_s\}(b\circ\Phi_s) + (b\circ\Phi_s)\partial_s G_s\,d\nu.
\end{equation}
To rewrite the first term on the right-hand side above, let us extend $(b\circ \Phi_s)G_s$ to a function on $T^*(\R\times M)$ which is fiber-homogeneous of degree 1 outside a small neighborhood of the zero section. This can be accomplished by choosing some $\chi\in C^\infty(\R)$ which vanishes in a neighborhood of zero and is equal to one outside a slightly larger neighborhood. Then, for any fixed $s\in\R,$ 
\[\chi(|\xi|)(b\circ\Phi_s)G_s\in S^1_{c\ell}(T^*(\R\times M)),\] 
and hence 
\[B := \Op(\chi(|\xi|)(b\circ\Phi_s)G_s) \in \Psi_{c\ell}^1(M).\]
Thus, $[B,P] \in \Psi_{c\ell}^2(M)$, and hence
\[\lim\limits_{k\to\infty}\langle[P,B]u_k,u_k\rangle \,\,\,= \!\!\! \int\limits_{S^*(\R\times M)}\untab\frac{1}{i} \{p,(b\circ\Phi_s)G_s\}\,d\nu\]
by the definition of the defect measure. On the other hand, since each $u_k$ solves the damped wave equation, we also have
\begin{align*}
\langle[P,B]u_k,u_k\rangle & = \langle B u_k,Pu_k\rangle - \langle Pu_k,B u_k\rangle \\
& = \langle B u_k,-2W\partial_t u_k\rangle + \langle 2W\partial_tu_k,B u_k \rangle\\
& = \langle 2(\partial_t WB + BW\partial_t)u_k,u_k\rangle.
\end{align*}
Taking the limit of both sides as $k \ra \infty$ gives
\[
\lim_{k\ra \infty} \langle[P,B]u_k,u_k\rangle = \int\limits_{S^*(\R\times M)}\untab 4 i\tau w (b\circ \Phi_s)G_s\,d\nu.
\]
Therefore,
\[\int\limits_{S^*(\R\times M)}\untab\{p,(b\circ\Phi_s)G_s\}\,d\nu \,\,\, =  -\untab\int\limits_{S^*(\R\times M)}\untab  4 \tau w (b\circ \Phi_s)G_s\,d\nu.\]
Combining the above with \eqref{propagation_product}, we obtain
\[\partial_s\untab \int\limits_{S^*(\R\times M)}\untab (b\circ \Phi_s)G_s\,d\nu\,\,\, = \!\!\!\int\limits_{S^*(\R\times M)}\untab (b\circ \Phi_s)\lp \partial_s G_s -\{p,G_s\} - 4\tau w G_s\rp\,d\nu,\]
which is clearly zero if $G_s$ satisfies \eqref{G_s_defn}.
\end{proof}

Another important observation about the defect measure $\nu$ is its support is closely related to the characteristic set of $P$, defined as 
\[\charset(P) = \{(t,x,\tau,\xi)\in T^*(\R\times M):\, p(t,x,\tau,\xi) = 0\}.\]
The following is a well-known result, but we provide a short proof for the benefit of the reader. 

\begin{lemma}
Given $\{u_k\}$ and $\nu$ as above, the support of $\nu$ is contained in the intersection $\charset(P)\cap S^*(\R\times M)$. 
\end{lemma}

\begin{proof}
First, let $\chi\in C^\infty(\R\times M)$ be supported in a neighborhood of $\charset(P)$ and identically one on a slightly smaller neighborhood. Then, since $P$ is elliptic on the support of $1 - \chi$, there exist a parametrix $Q\in \Psi_{c\ell}^{-2}(\R\times M)$ such that 
\[(\Id-\Op(\chi))u_k = QP u_k + Ru_k = Ru_k\]
for some smoothing operator $R$ and all $k$. Now fix an interval $I$ and let $\psi\in C_0^\infty(I)$. Since $R$ is smoothing,
\[
\|\psi(t)(I-\Op(\chi))u_k\|_{H^1(\R\times M)} = \|\psi(t)Ru_k\|_{H^1(I\times M)}\le \|u_k\|_{L^2(I\times M)}.
\]
By assumption, $u_k$ converges weakly to zero in $H^1(\R\times M)$, and therefore its restriction to $I\times M$ converges weakly to zero in $H^1(I\times M)$. Thus, there exists a subsequence $\{u_{k_j}\}$ which coverges strongly to zero in $L^2(I\times M)$, since $I\times M$ is compact. Therefore, $\|u_{k_j}\|_{L^2(I\times M)}\to 0$, which implies that $\psi(t)(\Id - \Op(\chi))u_{k_j} \to 0$ strongly in $H^1(I\times M)$ . Now consider the operator 
\[(\partial_t^2+ \Delta_g)\left[\psi(t)(I-\Op(\chi))\right]\in \Psi^2_{c\ell}(\R\times M),\]
which is supported in $I\times M.$ Since extracting a subsequence does not change the defect measure, we must have
\begin{equation}\label{defectcutoff}\langle(\partial_t^2+\Delta_g)\left[\psi(t)(\Id - \Op(\chi))u_{k_j}\right],u_{k_j})\rangle \to -\frac{1}{2}\!\!\!\!\!\int\limits_{S^*(\R\times M)}\!\!\!\!\!\psi(t)(1-\chi(t,x,\tau,\xi))\,d\nu,
\end{equation}
recalling that $\tau^2 + |\xi|_g^2 = \frac{1}{2}$ on $S^*(\R\times M).$ However, we also have that 
\begin{align*}
 & \|(\partial_t^2 + \Delta_g)\left[\psi(t)(\Id - \Op(\chi))\right]u_{k_j}\|_{H^{-1}(I\times M)} \\
 & \hskip 1in\le C\|\psi(t)(\Id - \Op(\chi))u_{k_j}\|_{H^1(I\times M)}.
\end{align*}
Recalling that $\psi(t)(I-\Op(\chi))u_{k_j}\to 0$ strongly in $H^1(I\times M)$ and noting that $\{u_{k_j}\}$ is bounded in $H^{1}(I\times M)$, we must have
\[\langle(\partial_t^2 + \Delta_g)\left[\psi(t)(\Id - \Op(\chi))u_{k_j}\right],u_{k_j})\rangle\to 0.\]
Combining this with \eqref{defectcutoff}, we see that the support of $\nu$ must be disjoint from that of $\psi(t)(1-\chi(t,x,\tau,\xi))$. Since $\psi$ was arbitrary, and the argument holds for any $\chi$ with the appropriate support properties, we must have that the support of $\nu$ is contained in $\charset(P)\cap S^*(\R\times M).$ 
\end{proof}

Observe that $p = 0$ exactly when $\tau = \pm |\xi|_g$, and therefore $\charset(P)\cap S^*(\R\times M)$ is comprised of the two connected components 
\[S^\pm = \{\tau = \mp 1/2\}\cap S^*(\R\times M).\]
It is helpful to define $\nu^\pm$ and $G_s^\pm$ to be the restrictions of $\nu$ and $G_s$ to $S^\pm$, respectively. The definition of $G_s$ given in \eqref{G_s_defn} then implies that 
\begin{equation}\label{G_pm_eqn}
\partial_s G_s^\pm = \{p,G_s^\pm\} \mp 2 w G_s^\pm.
\end{equation}
Now, since $w$ depends only on $(x,\xi)$ and $\tau$ is constant on $S^+$, we may treat $G_s^\pm$ as functions on $S^*M$. For the purposes of this article, it suffices to consider only $G_s^+$. It follows immediately from \lemref{defect_measure_propagation} that $G_s^+$ gives the propagation of $\nu^+$ on $S^+$ under the flow $\Phi_s$ i.e. $(\Phi_s)_*\nu^+ = G_{-s}^+\nu^+.$

As in \cite{Klein2017}, we claim that $G_s^+$ can be realized as the solution of a much simpler differential equation by observing that it is a cocycle map. That is, for any $(x,\xi)\in S^*M$ and any $r,s\in\R$, we have $G_{s+r}^+(x,\xi) = G_r^+(\varphi_s(x,\xi)) G_s^+(x,\xi).$ To see this, note that by \eqref{G_pm_eqn} and properties of the Poisson bracket,
\begin{align*}
\partial_s\lp(G_r^+\circ\varphi_s)G_s^+\rp & = (G_r^+\circ\varphi_s)\lp\{p,G_s^+\} - 2w G_s^+\rp + \{p,G_r^+\circ\varphi_s\}G_s^+\\
& = \{p,(G_r^+\circ\varphi_s)G_s^+\} - 2w(G_r^+\circ\varphi_s)G_s^+.
\end{align*}
So, $(G_r^+\circ \varphi_s)G_s^+$ and $G_{s+r}^+$ both satisfy the same initial value problem, and must be equal. Using the cocycle property and the fact that $G_0^+ \equiv 1$, we have 
\begin{align*}
\partial_s G_s^+ &= \lim\limits_{h\to 0} \frac{G_{s+h}^+ - G_s^+}{h} = \lim_{h \ra 0} \frac{(G_h^+\circ\varphi_s)G_s^+ - G_s^+}{h} \\
&=G_s^+ \lim_{h \ra 0} \frac{G_h^+ \circ \varphi_s -G_0^+ \circ \varphi_s}{h}= G_s^+\partial_r \lp G_r^+\circ\varphi_s\rp\big|_{r=0}.
\end{align*}
Since $G_0^+ \equiv 1$, we have that $\{p,G_r^+\}|_{r=0} =0$. This along with the fact that $\varphi_s$ is independent of $r$ gives $\partial_r(G_r^+\circ \varphi_s)\big|_{r=0} = -2w\circ\varphi_s.$ Thus, $G_s^+$ can be realized as the solution of the initial value problem
\[\begin{cases}
G_0^+(x,\xi) = 1\\
\partial_s G_s^+(x,\xi) = -2w(\varphi_s(x,\xi))G_s^+(x,\xi),
\end{cases}\]
which has solution
\begin{equation}\label{G_t_plus}
G_s^+(x,\xi) = \exp\lp-\int\limits_0^s 2w(\varphi_r(x,\xi))\,dr\rp.
\end{equation} 
Thus, the propagation of the defect measure exhibits exponential decay in proportion to the amount of time geodesics spend in the region where $w(x,\xi)$ is positive. 

\section{Pseudodifferential Operator Acting on Coherent States}\label{pseudocoherentsection}
A key component of the proof of \thmref{best_const_thm} is to build quasimodes for \eqref{damped_wave} using Guassian beams, which are strongly localized along a given geodesic. In this section, we obtain precise estimates for pseudo-differential operators acting on slightly simpler objects, namely coherent states. A coherent state on $\R^n$ is a sequence of smooth functions $\{h_k\}$ taking the form
\begin{equation}\label{coherent_state}
h_k(x) = k^{\frac{n}{4}}e^{ik\langle x-x_0,\xi_0\rangle}e^{\frac{ik}{2}\langle A(x-x_0),(x-x_0)\rangle}b(x)
\end{equation}
for some fixed $(x_0,\xi_0)\in S^*\R^{n}$, where $b\in C_c^\infty(\R^n)$ and $A\in \C^{n\times n}$ has positive definite imaginary part. Heuristically, one thinks of $h_k$ as being strongly microlicalized near $(x_0,\xi_0)$. The objective of this section is to show that if a symbol $a\in S^m_{c\ell}(\R^{2n})$ vanishes to some finite order at $(x_0,\xi_0)$, then $\|\Op(a)h_k\|_{L^2(\R^n)}$ satisfies a bound which depends on the symbol order $m$ and on the order of vanishing. 

\begin{rmk}
\textnormal{
For the purposes of this section only, we use the more standard convention that $S^*\R^n = \{(x,\xi)\in\R^{2n}:\,|\xi| = 1\}$ for the sake of convenience, but this does not alter the analysis in any way. 
}
\end{rmk}
\begin{prop}\label{gaussian_beam_prop}
Fix $(x_0,\xi_0)\in S^*\R^{n}$, $b\in C_c^\infty(\R^n)$, and a matrix $A\in\Cb^{n\times n},$ with positive definite imaginary part. Then, for any $k \ge 1,$ let $h_k$ be given by \eqref{coherent_state}. Let $a\in S_{c\ell}^{m}(\R^{2n})$ have compact support in $x$ and a polyhomogeneous expansion given by
\[a\sim\sum\limits_{j\ge 0} a_{m-j},\]
where each $a_{m-j}\in S_{c\ell}^{m-j}(\R^{2n})$ satisfies $a_{m-j}(x,s\xi) = s^{m-j}a_{m-j}(x,\xi)$ for all $s > 0$ and $|\xi|\ge c > 0$ for some small $c$. Suppose there exists an $\ell\in\N$ such that $a_{m-j}$ vanishes to order $\ell -2j$ at $(x_0,\xi_0)$ for all $j\le \frac{\ell}{2}$. Then, for each $\ve >0$, there exists a $C_\ve > 0$ so that
\begin{equation}
\|\Op(a)h_k\|_{L^2(\R^n)} \le C_\ve k^{m-\frac{\ell}{2}+\ve}.
\end{equation}
\end{prop}

\begin{proof}
By the polyhomogeneity of $a$, for any $N_0 \geq 0$ there exists $r_{N_0}\in S_{c\ell}^{m-N_0}$ such that 
\begin{equation}\label{a_expansion}
a = \sum_{j=0}^{N_0-1} a_{m-j}+r_{N_0}.
\end{equation}
We begin with the following lemma, which handles the remainder term in this expansion. 
\begin{lemma}\label{remainder_lemma}
Let $r\in S_{c\ell}^{-s}(R^{2n})$ with $s \ge 0.$ Then, there exists a $C> 0$ such that 
\begin{equation}\label{remainder_bound}
\|\Op(r)h_k\|_{L^2(\R^{n})} \leq Ck^{-\frac{s}{2}}.
\end{equation}
\end{lemma}
\begin{proof}
By the quantization formula, we have 
\[\|\Op(r)h_k\|_{L^2(\R^{n})}^2 =  \int\limits_{\R^n}\left|\int\limits_{\R^{2n}}e^{i\langle x- y,\xi\rangle}r(x,\xi)h_k(y)\,dy\,d\xi\right|^2\,dx. \]
Assume without loss of generality that $x_0 = 0$. We then change variables via $x \mapsto k^{-\frac{1}{2}}x$, $y \mapsto k^{-\frac{1}{2}}y$, and $\xi\mapsto k^{\frac{1}{2}}\xi.$ Recalling the definition of $h_k$, we obtain
\begin{equation}\label{remainder_term}
\|\Op(r) h_k \|^2_{L^2(\R^n)} = \int\limits_{\R^n}\left|\int\limits_{\R^{2n}}e^{i\langle x- y,\xi\rangle}r(k^{-\frac{1}{2}}x,k^{\frac{1}{2}}\xi)e^{ik^{\frac{1}{2}}\langle x,\xi_0 \rangle}e^{\frac{i}{2}\langle Ay,y\rangle}b(k^{-\frac{1}{2}}y)\,dy\,d\xi\right|^2\,dx.
\end{equation}
For notational convenience, we define $g_A(y) = e^{\frac{i}{2}\langle A y,y\rangle}$ and let $\tau_s:C^\infty(\R^n)\to C^\infty(\R^n)$ denote dilation by $s > 0$. That is, $\tau_s f(y) = f(sy).$ Then, using $\hat{\empty}$ to denote the standard Fourier transform, we define  
\begin{equation}\label{F_k}
F_k(\xi):=\int_{\Rn} e^{-i\<y,\xi\>} g_A(y) b(k^{-1/2} y) dy= k^{n/2}[\wh g_A\ast\tau_{k^{\frac{1}{2}}}\wh b](\xi).
\end{equation}
Thus, we can rewrite \eqref{remainder_term} as 
\begin{equation}\label{remainder_FT}
\|\Op(r)h_k\|_{L^2(\R^n)}^2 = \int\limits_{\R^n}\left|\int\limits_{\R^{n}}e^{i\langle x,\xi\rangle}r(k^{-\frac{1}{2}}x,k^{\frac{1}{2}}\xi) F_k(\xi - k^{\frac{1}{2}}\xi_0)d\xi\right|^2\,dx.
\end{equation}
We claim that for any $N\in\N$ and any multi-index $\beta,$ there exists a constant $C_{N,\beta} > 0$ so that 
\begin{equation}\label{schwartz_bound}
\left|\partial_\xi^\beta F_k(\xi)\right| \le C_{N,\beta} (1+|\xi|)^{-N} \quad \text{for all }k\in\N.
\end{equation}
To see this, consider the case where $|\beta| = 0$ and note that 
\begin{align*}
\left| (1 + |\xi|)^{N}k^{n/2}[\wh g_A\ast\tau_{k^{\frac{1}{2}}}\wh b](\xi)\right| &\le C_N k^{n/2}\int\limits_{\R^n}(1 + |\xi-\eta|)^{N}(1+|\eta|)^{N}|\wh g_A(\xi-\eta)|\,|\wh b(k^{\frac{1}{2}}\eta)|\,d\eta\\
& \le C_{N}k^{n/2}\int\limits_{\R^n}(1+|\eta|)^{N}|\wh b(k^{\frac{1}{2}}\eta)|\,d\eta,
\end{align*}
where the last inequality follows from the fact that $\wh g_A$ is Schwartz class, since $A$ has positive definite imaginary part. Now, observe that 
\[k^{n/2}\int\limits_{\R^n}(1 + |\eta|)^{N}|\wh b(k^{\frac{1}{2}}\eta)\,d\eta|\le k^{n/2}\int\limits_{\R^n}\frac{(1+k^{\frac{1}{2}}|\eta|)^{N+n+1}}{(1+k^{\frac{1}{2}}|\eta|)^{n+1}}|\wh b(k^{\frac{1}{2}}\eta)|\,d\eta \le C_Nk^{n/2}\int\limits_{\R^n}(1+k^{\frac{1}{2}}|\eta|)^{-n-1}\,d\eta,\]
for some new $C_N>0$, since $\wh b$ is Schwartz-class. Changing variables via $\eta\mapsto k^{-\frac{1}{2}}\eta$, we obtain that
\[\left| (1 + |\xi|)^{N}k^{n/2}[\wh g_A\ast\tau_{k^{\frac{1}{2}}}\wh b](\xi)\right| \le C_N k^{n/2}\]
after potentially increasing $C_N.$ Dividing through by $k^{n/2}(1+|\xi|)^N$ completes the proof of \eqref{schwartz_bound} for $|\beta| = 0$. To obtain the estimate when $|\beta|\ne 0,$ simply repeat the above proof with $\wh g_A$ replaced by $\partial_\xi^\beta \wh g_A.$

Now, in order to estimate \eqref{remainder_FT} we introduce a smooth cutoff function $\chi$ which is identically one in a neighborhood of $x=0$. We then write 
\[\|\Op(r)h_k\|_{L^2(\R^n)}^2 = I + \II,\]
where $I$ is defined by
\[
I = \int\limits_{\R^n}\left|\int\limits_{\R^{n}}e^{i\langle x,\xi\rangle}\chi(x)r(k^{-\frac{1}{2}}x,k^{\frac{1}{2}}\xi) F_k(\xi - k^{\frac{1}{2}}\xi_0)d\xi\right|^2\,dx,
\]
and $\II$ is defined analogously with $\chi(x)$ replaced by $1 - \chi(x)$. To estimate $I$, we note that when $|\xi| \le 1$,
\[|r(k^{-\frac{1}{2}}x,k^{\frac{1}{2}}\xi)F_k(\xi - k^{\frac{1}{2}}\xi_0)|\le C_{N}(1+|\xi-k^{\frac{1}{2}}\xi_0|)^{-N}\le C_{N}'k^{-N/2},\]
for some $C_N,\,C_N' >0$ and any $N$, by \eqref{schwartz_bound} and the fact that $r$ has nonpositive order and is therefore uniformly bounded. Thus,
\begin{equation}\label{chi_op_eqn}
 I = \int\limits_{\R^n}\left|\int\limits_{|\xi|\ge 1}e^{i\langle x,\xi\rangle}\chi(x)r(k^{-\frac{1}{2}}x,k^{\frac{1}{2}}\xi) F_k(\xi-k^{\frac{1}{2}}\xi_0)d\xi\right|^2\,dx +\mathcal O(k^{-\infty}).
\end{equation}
Now, when $|\xi|\ge 1$,  
\[|r(k^{-\frac{1}{2}}x,k^{\frac{1}{2}}\xi)|\le C (1 + k^{\frac{1}{2}}|\xi|)^{-s}\le  C k^{-\frac{s}{2}}.\]
Combining this with \eqref{chi_op_eqn}, we have
\begin{equation}\label{eq:chirn}
I\le C k^{-s}\|F_k\|_{L^1(\R^n)}^2 \le C' k^{-s},
\end{equation}
where the final inequality follows from \eqref{schwartz_bound}.

Now, consider $\II$. Since $1-\chi$ vanishes in a neighborhood of $x = 0$, we may integrate by parts arbitrarily many times in $\xi$ using the operator $\frac{\langle x,\nabla_\xi\rangle}{i|x|^2}$, which preserves $e^{i\langle x, \xi \rangle}$. That is, for any $\nu \geq 0$, we have
\[\II = \int\limits_{\R^n}\left|\int\limits_{\R^{n}}e^{i\langle x,\xi\rangle}\lp 1-\chi(x)\rp \lp\frac{i\langle x,\nabla_\xi\rangle}{|x|^2}\rp^{\nu}\lp r(k^{-\frac{1}{2}}x,k^{\frac{1}{2}}\xi) F_k(\xi-k^{\frac{1}{2}}\xi_0)\rp d\xi\right|^2\,dx.\]
By \eqref{schwartz_bound} and the fact that $r \in S_{c\ell}^{-s}$, we have for any multi-index $\beta$ and any $N$,
\[\left|\partial_\xi^\beta\lp r(k^{-\frac{1}{2}}x,k^{\frac{1}{2}}\xi)F_k(\xi-k^{\frac{1}{2}}\xi_0)\rp\right|\le \sum\limits_{|\gamma|\le |\beta|} C_{\gamma,N}k^{\frac{|\gamma|}{2}}(1+k^{\frac{1}{2}}|\xi|)^{-s-|\gamma|}(1+|\xi-k^{\frac{1}{2}}\xi_0|)^{-N}.\]
In the region where $|\xi| \ge 1$, the above is bounded by $C_{N}k^{-\frac{s}{2}}(1+|\xi-k^{\frac{1}{2}}\xi_0|)^{-N}$ for some $C_N>0$. Alternatively, when $|\xi|\le 1$, we have a bound of the form $C_{N}k^{-N/2}$, since ${1+|\xi-k^{\frac{1}{2}}\xi_0|\ge Ck^{\frac{1}{2}}}$. Combining these facts, we have 
\[\int\limits_{\R^n}\left|\int\limits_{\R^{n}}e^{i\langle x,\xi\rangle}\lp 1-\chi(x)\rp \lp\frac{i\langle x,\nabla_\xi\rangle}{|x|^2}\rp^{\nu}\lp r(k^{-\frac{1}{2}}x,k^{\frac{1}{2}}\xi) F_k(\xi-k^{\frac{1}{2}}\xi_0)\rp d\xi\right|^2\,dx \le C_Nk^{-s},\]
for some $C_N>0$, provided $\nu > \frac{n+1}{2}$ so that the integral in $x$ is convergent. Therefore,
$$
\II \leq C_N k^{-s}.
$$
Combining this with \eqref{eq:chirn} and taking square roots of both sides completes the proof.

\end{proof}

We now return to the proof of \propref{gaussian_beam_prop}. We aim to estimate each of the terms in the sum in \eqref{a_expansion} separately. By definition, for each $j\le N_0-1$,
\[
\|\Op(a_{m-j}) h_k\|_{L^2(\R^n)}^2 = \int_{\Rn}\left|\int_{\Rb^{2n}} e^{i\<x-y,\xi\>} a_{m-j}(x,\xi) h_k(y) dy d\xi\right|^2 dx.
\]

\noindent As before, we change variables via $x \mapsto k^{-\frac{1}{2}}x, y\mapsto k^{-\frac{1}{2}}y,$ and $ \xi \mapsto k^{\frac{1}{2}}\xi.$ This gives 
$$
\|\Op(a_{m-j}) h_k\|_{L^2(\R^n)}^2 =\int_{\Rn} \left| \int_{\Rb^{2n}} e^{i\<x-y,\xi\>}a_{m-j}(k^{-\frac{1}{2}}x,k^{\frac{1}{2}}\xi) F_k(\xi-k^{\frac{1}{2}}\xi) \right|^2dx,
$$
where $F_k$ is given by \eqref{F_k}. Now, let $\chi(x,\xi) = \chi(\xi)$ be a smooth function which is identically one for $|\xi|\le \frac{1}{2}$ and zero outside $|\xi|\le 1$. Then, for any $0<\alpha<1/2$ and any $k>0$, define
\[
\chi_{k,\alpha}(\xi) = \chi(k^{-\alpha}(\xi-k^{\frac{1}{2}}\xi_0)),
\]
so that $\chi_{k,\alpha}$ is identically one on the ball of radius $\frac{1}{2}k^{\alpha}$ centered at $k^{\frac{1}{2}} \xi_0$ and zero outside the corresponding ball of radius $k^\alpha$. Since $\chi_{k,\alpha}$ is supported on $|\xi|>2$ for sufficiently large $k$, the homogeneity of $a_{m-j}$ implies
\begin{equation}\label{a_j_homog}
\chi_{k,\alpha}(\xi) a_{m-j}(k^{-\frac{1}{2}}x, k^{\frac{1}{2}} \xi) = k^{\frac{m-j}{2}} \chi_{k,\alpha}(\xi) a_{m-j}(k^{-\frac{1}{2}}x, \xi). 
\end{equation}
Recall that $a_{m-j}$ vanishes to order $\ell_j:=\ell-2j$ at $(x_0=0,\xi_0)$ for all $j \leq  \frac{\ell}{2}$, so by Taylor expansion, there exists a collection of $f_{m-j,\gamma},\,g_{m-j,\gamma}\in S_{c\ell}^{m-j}(\R^{2n})$ such that
\[a_{m-j}(x,\xi)= \sum\limits_{|\gamma|=\ell_j}x^{\gamma} f_{m-j,\gamma}(x,\xi) + \lp\frac{\xi}{|\xi|}-\xi_0\rp^{\gamma} g_{m-j,\gamma}(x,\xi).\]
Combining this with \eqref{a_j_homog}, we have
\[
\chi_{k,\alpha}(\xi) a_{m-j}(k^{-\frac{1}{2}}x, k^{\frac{1}{2}} \xi) = k^{\frac{m-j}{2}}\sum\limits_{|\gamma|= \ell_j} \left( k^{-\frac{\ell_j}{2}}x^{\gamma} f_{m-j,\gamma}(k^{-\frac{1}{2}}x,\xi) + \left( \frac{\xi}{|\xi|}-\xi_0\right)^{\gamma} g_{m-j,\gamma}(k^{-\frac{1}{2}}x,\xi) \right).
\]
Then, we define
\begin{align}
\begin{split}\label{A_j_1_defn}
&\mathcal{A}_{j,1}(x) = k^{\frac{m-j-\ell_j}{2}} \sum\limits_{|\gamma|=\ell_j}\intrn e^{i\<x,\xi\>} \chi_{k,\alpha}(\xi)x^{\ell_j} f_{m-j,\gamma}(k^{-\frac{1}{2}}x,\xi) F_k(\xi-k^{\frac{1}{2}} \xi_0) d\xi, 
\end{split}\\
\begin{split}\label{A_j_2_defn}
&\mathcal{A}_{j,2}(x) = k^{\frac{m-j}{2}} \sum\limits_{|\gamma|=\ell_j}\intrn  e^{i\<x,\xi\>} \chi_{k,\alpha}(\xi)\left(\frac{\xi}{|\xi|}-\xi_0\right)^{\gamma} g_{m-j,\gamma}(k^{-\frac{1}{2}}x,\xi) F_k(\xi-k^{\frac{1}{2}} \xi_0) d\xi. 
\end{split}
\end{align}
and
\begin{equation}\label{R_j_definition}
\mathcal R_j(x) =\int\limits_{\R^{n}}e^{i\langle x,\xi\rangle}\lp1-\chi_{k,\alpha}(\xi)\rp a_{m-j}(k^{-\frac{1}{2}},k^{\frac{1}{2}}\xi)F_k(\xi-k^{\frac{1}{2}}\xi)\,d\xi,
\end{equation}
so that
$$
\Op(a_{m-j})h_k =  \mathcal{A}_{j,1} + \mathcal{A}_{j,2} + \mathcal R_j.
$$
We claim that $\mathcal R_j$ is negligible for large $k$. Since $F_k(\xi-k^{\frac{1}{2}}\xi_0)$ is a Gaussian centered at $k^{\frac{1}{2}}\xi_0$ and $1-\chi_{k,\alpha}$ is supported at least $k^{\alpha}$ away from that center, we are able to show that $\mathcal{R}_j$ is controlled by an arbitrarily negative power of $k$.
\begin{lemma}\label{R_j_lemma}
For any $N' \in \mathbb{N}$ there exists $C_{N'}>0$ such that
\begin{equation}\label{R_j_L2}
\|\mathcal R_j\|_{L^2(\R^n)} \le C_{N'}k^{-N'}.
\end{equation}
\end{lemma}
\begin{proof}
To begin note that for any multi-index $\beta$
\begin{equation}\label{chi_deriv}
\partial_\xi^\beta\chi_{k,\alpha}(\xi) = 
k^{-\alpha|\beta|}(\partial_\xi^\beta\chi)(k^{-\alpha}(\xi-k^{\frac{1}{2}}\xi_0)).
\end{equation}
Combining \eqref{chi_deriv} with \eqref{schwartz_bound} shows that for any $N \in \mathbb{N}$ and any multi-index $\beta$, there exists $C_{N,\beta}>0$ such that 
\[\partial_\xi^\beta\left[\mathcal(1-\chi_{k,\alpha}) F_k(\xi)\right]\le C_{N,\beta}k^{-\alpha|\beta|}\mathds 1_{\supp(1-\chi_{k,\alpha})}(\xi)(1 + |\xi-k^{\frac{1}{2}}\xi_0|)^{-N},\]
where for any set $E\subset\R^{n}$, $\mathds 1_{E}$ denotes the indicator function of $E$.

Now, when $|x| \ge 1$, for $\nu \geq 0$ we may integrate by parts in \eqref{R_j_definition} as in the proof of \lemref{remainder_lemma} to obtain
\[
\mathcal{R}_j(x) =\int_{\Rn} e^{i\<x,\xi\>}\lp\frac{i\langle x,\nabla_\xi\rangle}{|x|^2}\rp^\nu\left[(1-\chi_{k,\alpha}(\xi)) a_{m-j}(k^{-\frac{1}{2}}x,k^{\frac{1}{2}}\xi) F_k(\xi)\right] d\xi.
\]
Since $a_{m-j} \in S_{c\ell}^{m-j}(\R^{2n})$, we have that for any multi-index $\beta$,
\[|\p_{\xi}^{\beta} a_{m-j}(k^{-\frac{1}{2}}x, k^{\frac{1}{2}} \xi)|\leq Ck^{\frac{|\beta|}{2}} (1+k^{\frac{1}{2}}|\xi|)^{m-j-|\beta|} \leq C k^{\frac{|\beta|}{2}} (1+k^{\frac{1}{2}} |\xi|)^{m}\] 
Thus, for any $N \in \mathbb{N}$, there exists a constant $C_N$ such that whenever $|x| \ge 1$,
$$
\left|\mathcal{R}_j(x)\right| \leq C_N \sup_{|\beta|\le \nu}\frac{1}{|x|^{\nu}}\int\limits_{\R^n} \mathds 1_{\supp(1-\chi_{k,\alpha})}(\xi)(1+|\xi-k^{\frac{1}{2}}\xi_0|)^{-N} k^{\frac{|\beta|}{2}} (1+k^{\frac{1}{2}}|\xi|)^{m} \,d\xi.
$$
Recall that $|\xi_0|=1$, and so by the triangle inequality
$$
1+k^{\frac{1}{2}}|\xi| \leq 1+k+k^{\frac{1}{2}} |\xi-k^{\frac{1}{2}}\xi_0| \leq C k(1+|\xi-k^{\frac{1}{2}} \xi_0|).
$$
Thus,
\begin{equation}\label{A_j_max}
|\mathcal{R}_j(x)| \leq C_N \sup_{|\beta| \leq \nu} \frac{1}{|x|^{\nu}}\int\limits_{\R^n}\mathds 1_{\supp(1-\chi_{k,\alpha})}(\xi)(1+|\xi-k^{\frac{1}{2}}\xi_0|)^{-N+m} k^{m+\frac{|\beta|}{2}} \,d\xi.
\end{equation}

\noindent Using polar coordinates $\xi-k^{\frac{1}{2}}\xi_0 = r\omega$ with $r\in \R^+,$ $\omega\in \mathbb S^{n-1}$, we compute
\begin{align*}
\intrn \mathds 1_{\supp(1-\chi_{k,\alpha})}(\xi)(1 + |\xi-k^{\frac{1}{2}}\xi_0|)^{-N+m}\,d\xi & \le \int\limits_{\mathbb S^{n-1}}\int\limits_{k^\alpha/2}^\infty(1 + r)^{-N+m}r^{n-1}\,dr\,d\omega\\
&\le Ck^{\alpha(m+n-N)}.
\end{align*}
Combining this with \eqref{A_j_max}, we have
\[\left|\mathcal R_j(x)\right|\le C_N|x|^{-\nu} k^{\alpha(m+n-N)+m+\frac{|\beta|}{2}},\quad \text{if }|x|\ge 1.\]
Since $\nu$ and $N$ were both arbitrary, given any $N'\ge 0$ we can choose $\nu \ge \frac{n+1}{2}$ and $N$ sufficiently large so that 
\[\left|\mathcal R_j(x)\right|\le C_{N'}k^{-N'}|x|^{-\frac{n+1}{2}},\quad \text{ if }|x| \ge 1.\]
By an analogous argument when $|x| \le 1$, except without integration by parts, we have
\[\left|\mathcal R_j(x)\right|\le C_{N'}k^{-N'}\quad \text{ if }|x|\le 1.\]
Combining these inequalities and taking the $L^2$ norm completes the proof of \eqref{R_j_L2}.
\end{proof}

It remains to estimate $\mathcal{A}_{j,1}$ and $\mathcal{A}_{j,2}$. It is here that we take advantage of the compatibility of the vanishing of $a_{m-j}$ with the particular form of the coherent state $h_k$. We first consider $\mathcal{A}_{j,1}$.
\begin{lemma}\label{A_1_lemma}
For any $j\ge 0$, there exists $C_j>0$ such that 
\begin{equation}\label{A_1_bound}
\|\mathcal{A}_{j,1}\|_{L^2(\R^n)} \le C_j k^{m-j-\frac{\ell_j}{2}}.
\end{equation}
\end{lemma}
\begin{proof}
Note on the support of $\chi_{k,\alpha}$ 
\[k^{\frac{1}{2}}-k^{\alpha} \leq |\xi| \leq k^{\frac{1}{2}}+k^{\alpha}.\]
Also, recall that $f_{m-j,\gamma} \in S_{c\ell}^{m-j}(\R^{2n})$, so $|\p_{\xi}^{\beta} f_{m-j,\gamma}(x,\xi)| \leq C_{\beta} |\xi|^{m-j-|\beta|}$. Therefore,
\begin{equation}\label{f_bound}
\sup\limits_{x\in\R^n}|\p_{\xi}^{\beta} f_{m-j,\gamma}(x,\xi)| \leq C_\beta k^{\frac{m-j-|\beta|}{2}}, \quad \text{ for all }\xi\in \supp\chi_{k,\alpha}.
\end{equation}
Now, when $|x|\ge 1$, we may integrate by parts as before to obtain that for any $\nu\ge 0$ and any $N$ sufficiently large,
\begin{align*}
\left|\mathcal{A}_{j,1}(x)\right|& \le k^{\frac{m-j-\ell_j}{2}}\sum\limits_{|\gamma| = \ell_j}\left|\intrn e^{i\<x,\xi\>}\chi_{k,\alpha}(\xi) x^{\gamma} f_{m-j,\gamma}(k^{-\frac{1}{2}}x,\xi) F_k(\xi-k^{\frac{1}{2}} \xi_0) d\xi\right|\\
&\le k^{\frac{m-j-\ell_j}{2}} |x|^{\ell_j}\intrn \left|\lp\frac{i\langle x,\nabla_\xi\rangle}{|x|^2}\rp^\nu\left[\chi_{k,\alpha}(\xi)f_{m-j,\gamma}(k^{-\frac{1}{2}}x,\xi)F_k(\xi-k^{\frac{1}{2}}\xi_0)\right]\right|\,d\xi\\
& \le C_{\nu,N} k^{m-j-\frac{\ell_j}{2}} |x|^{\ell_j - \nu}\intrn (1 + |\xi-k^{\frac{1}{2}}\xi_0|)^{-N}\,d\xi\\
& \le C_{\nu,N}' k^{m-j-\frac{\ell_j}{2}}|x|^{\ell_j-\nu},
\end{align*}
where the second-to-last inequality follows from \eqref{schwartz_bound}, \eqref{chi_deriv} and \eqref{f_bound}. When $|x| \leq 1$, by the same argument with $\nu =0$, we obtain
\[|\mathcal{A}_{j,1}(x)|\le C k^{m-j-\frac{\ell_j}{2}}.\]
Thus, for each $\nu \ge 0$, there exists a constant $C_\nu > 0$ so that 
\[|\mathcal{A}_{j,1}(x)| \le  C_\nu k^{m-j-\frac{\ell}{2}}(1+|x|)^{\ell_j-\nu} \quad \text{ for all }x\in\R^n.\]
Choosing $\nu$ so that $\ell_j-\nu \le -\frac{n+1}{2}$ and taking the $L^2$ norm gives the desired inequality.
\end{proof}
Finally, we turn our attention to $\mathcal{A}_{j,2}$. The estimation of this term is the most subtle of the three, and requires some very technical analysis of the relationship between the vanishing factor $\lp\frac{\xi}{|\xi|} - \xi_0\rp$ and the structure of the support of $\chi_{k,\alpha}$.
\begin{lemma} \label{A_2_lemma}
For any $j\ge 0,$ there exists $C_j>0$ such that
\begin{equation}\label{A_2_bound}
\left\|\mathcal{A}_{j,2}\right\|_{L^2(\R^n)}\le C_j k^{m-j+(\alpha-\frac{1}{2})\ell_j}.
\end{equation}
\end{lemma}
\begin{proof}
As before, we first consider the case where $|x| \ge 1$ and the case $|x| \leq 1$ will follow from an analogous argument. When $|x|\ge 1$, we may again use integration by parts to see that for any $\nu\ge 0$ and any $N$ sufficiently large,
$$
\left|\mathcal{A}_{j,2}(x)\right| \!\le\! k^{\frac{m-j}{2}}\!\!\!\sum\limits_{|\gamma|=\ell_j}\left|\,\int\limits_{\R^n}\!\!e^{i\langle x,\xi\rangle}\!\lp\frac{i\langle x,\nabla_\xi\rangle}{|x|^2}\rp^\nu\!\left[\chi_{k,\alpha}(\xi)\!\lp\!\frac{\xi}{|\xi|} - \xi_0\!\rp^\gamma\!\!\! g_{m-j,\gamma}(k^{-\frac{1}{2}}x,\xi)F_k(\xi-k^{\frac{1}{2}}\xi_0)\right]\!d\xi\right|.
$$

\noindent Note that on $\supp \chi_{k,\alpha},$ $k^{\frac{1}{2}}-k^{\alpha} \leq |\xi| \leq k^{\frac{1}{2}}+k^{\alpha}$, and so $|\p_{\xi}^{\beta} g_{m-j,\gamma}| \leq C k^{\frac{m-j-|\beta|}{2}} \leq Ck^{\frac{m-j}{2}}$.  Combining this with \eqref{schwartz_bound} gives
\begin{align}\label{Bj2_bound1}
\left|\mathcal{A}_{j,2}(x)\right| \le Ck^{m-j}|x|^{-\nu}\sum\limits_{\beta \leq \nu}\sum\limits_{|\gamma| = \ell_j}\int\limits_{\R^n}\left|\partial_\xi^{\beta}\lp\chi_{k,\alpha}(\xi)\lp\frac{\xi}{|\xi|}-\xi_0\rp^\gamma\rp\right|(1 + |\xi-k^{\frac{1}{2}}\xi_0|)^{-N}\,d\xi,
\end{align}
when $|x| \ge 1$. Therefore, it is sufficient to show that for any multi-indices $\beta,\,\gamma$ with $|\beta|\le \nu$ and $|\gamma|=\ell_j$, there exists $C>0$ such that
\begin{equation}\label{xi_factor}
\left|\partial_\xi^\beta\lp\chi_{k,\alpha}(\xi)\lp\frac{\xi}{|\xi|}-\xi_0\rp^\gamma\rp\right| \leq C k^{(\alpha-\frac{1}{2})\ell_j}.
\end{equation}
To show this, it is convenient to choose coordinates on $\R^n$ so that $\xi_0 = (1,0,\dotsc,0)$. 
Writing $\xi=(\xi_1, \xi_2, \ldots, \xi_n)$, we have that on the support of $\chi_{k,\alpha},$
\[|\xi - k^{1/2}\xi_0| = \sqrt{(\xi_1-k^{1/2})^2 + \xi_2^2+\cdots+\xi_n^2} \leq k^{\alpha},\]
by the definition of $\chi_{k,\alpha}$. Thus,
\begin{equation}\label{xisize}
|\xi_r| \leq k^{\alpha} \text{ for } i \neq 1\quad \text{ and }\quad|\xi|\geq k^{1/2}-k^{\alpha}.
\end{equation} 
Also, note that $\xi_1 > 0$ on $\supp\chi_{k,\alpha}$ for $k$ large enough, and so we can write 
\[|\xi| - \xi_1 = (|\xi| + \xi_1)^{-1}(\xi_2^2+\dotsm+\xi_n^2).\]
Combining these facts, we have that for large $k$, 
\[
|\xi|-\xi_1=\frac{\xi_2^2+\cdots+\xi_n^2}{|\xi|+\xi_1}\le \frac{(n-1)k^{2\alpha}}{k^{1/2}-k^{\alpha}} \leq C k^{2\alpha-\frac{1}{2}}.
\]
We can now show \eqref{xi_factor} for $\beta=0$. Recalling that $\gamma=(\gamma_1, \gamma_2, \ldots, \gamma_n)\in\N^{n}$ is a multi-index with $|\gamma|= \gamma_1+\dotsm+\gamma_n=\ell_j,$ we make use of the above inequality and \eqref{xisize} to obtain
\begin{align*}
 \left| \left( \frac{\xi}{|\xi|} -\xi_0 \right)^{\gamma} \right| & = \frac{ |\xi_1-|\xi||^{\gamma_1}|\xi_2|^{\gamma_2}\cdots |\xi_n|^{\gamma_n}}{|\xi|^{\ell_j}}\\
 &\leq C \frac{k^{2\alpha\gamma_1 }}{k^{\frac{\gamma_1}{2}}} \cdot\frac{k^{\alpha \gamma_2 } \cdots k^{\alpha\gamma_n }}{k^{\frac{\ell_j}{2}}}\\
 &= C k^{(\alpha-\frac{1}{2})\ell_j + (\alpha-\frac{1}{2})\gamma_1} \\
 &\leq C k^{(\alpha-\frac{1}{2})\ell_j},
\end{align*}
which proves \eqref{xi_factor} in the case where $|\beta| = 0$. To handle the case where $\beta \neq 0$ we let $\beta_1,\beta_2, \ldots,\beta_n$ be multi-indices and expand via the product rule to obtain
\begin{equation}\label{multiindexexpand}
\left| \p_{\xi}^{\beta}\lp\frac{\xi}{|\xi|}-\xi_0\rp^{\gamma} \right| \leq \sum_{|\beta_1+\beta_2+\cdots+\beta_n|=\beta} \!\!\!\!C_{\beta_j} \left|\p_{\xi}^{\beta_1} \lp\frac{\xi_1-|\xi|}{|\xi|}\rp^{\gamma_1} \p_{\xi}^{\beta_2} \lp\frac{\xi_2}{|\xi|}\rp^{\gamma_2}\cdots \p_{\xi}^{\beta_n}\lp\frac{\xi_n}{|\xi|}\rp^{\gamma_n}\right|.
\end{equation}
Since $\frac{\xi_1-|\xi|}{|\xi|}$ and $\frac{\xi_j}{|\xi|}$ are homogeneous of degree zero, we have that for any multi-index $\theta$,
\[
\left|\p_{\xi}^{\theta} \lp\frac{\xi_1-|\xi|}{|\xi|}\rp\right| \leq \frac{C}{|\xi|^{|\theta|}}, \quad\text{ and }\quad \left|\p_{\xi}^{\theta} \lp \frac{\xi_r}{|\xi|} \rp \right| \leq \frac{C}{|\xi|^{|\theta|}}\, \text{ if }r \neq 1.
\]
Furthermore, we recall that  $|\xi| \geq k^{\frac{1}{2}}-k^{\alpha}\geq C k^{\frac{1}{2}}$ on $\supp \chi_{k,\alpha},$ and so 
\begin{equation}\label{homogeneousestimate}
 \left|\chi_{k,\alpha}(\xi)\p_{\xi}^{\theta} \lp\frac{\xi_1-|\xi|}{|\xi|}\rp \right| \leq C k^{-\frac{|\theta|}{2}}, \quad\text{ and }\quad  \left|\chi_{k,\alpha}(\xi)\p_{\xi}^{\theta} \lp \frac{\xi_r}{|\xi|} \rp \right| \leq C k^{-\frac{|\theta|}{2}}, \, r \neq 1.
\end{equation}
\noindent Now consider $\p_{\xi}^{\beta_r} \lp\frac{\xi_r}{|\xi|}\rp^{\gamma_r}$. Expanding via the product rule, we can rewrite this as a linear combination of terms of the form
\[
\lp \frac{\xi_r}{|\xi|}\rp^{t_r} \p_{\xi}^{\delta_1}\lp \frac{\xi_r}{|\xi|}\rp \cdots \p_{\xi}^{\delta_q}\lp \frac{\xi_r}{|\xi|}\rp,
\]
where each $\delta_i$ is a multi-index with $|\delta_i| \geq 1$, $\delta_1 + \cdots +\delta_q = \beta_r,$ and $t_r + q = \gamma_r.$ That is, there are $t_r$ factors of $\frac{\xi_r}{|\xi|}$ which do not have any derivatives, and the remaining $\gamma_r-t_r$ factors each have at least 1 derivative applied to them. Note then that $t_r \geq \max(0,\gamma_r-|\beta_r|).$ 
By \eqref{xisize} each factor with no derivatives is bounded by $k^{\alpha-\frac{1}{2}}$. The homogeneous estimate \eqref{homogeneousestimate} controls the factors with derivatives, giving 
\[
\left|\chi_{k,\alpha}(\xi) \lp \frac{\xi_r}{|\xi|}\rp^{t_r} \p_{\xi}^{\delta_1}\lp \frac{\xi_r}{|\xi|}\rp \cdots \p_{\xi}^{\delta_q}\lp \frac{\xi_r}{|\xi|}\rp \right| \leq k^{(\alpha-\frac{1}{2})t_r} k^{-\frac{|\delta_1|}{2}} \cdots k^{-\frac{|\delta_q|}{2}} \leq C k^{(\alpha-\frac{1}{2})t_r -\frac{|\beta_r|}{2}}.
\]
Thus, by the triangle inequality, we have
\[
 \left|\chi_{k,\alpha}(\xi)  \p_{\xi}^{\beta_r} \left( \frac{\xi_r}{|\xi|} \right)^{\gamma_r} \right| \leq  C k^{(\alpha-\frac{1}{2})t_r -\frac{|\beta_r|}{2}}.
\]
When $\gamma_r -|\beta_r|>0$, we have
\[
 k^{(\alpha-\frac{1}{2})t_r} k^{-\frac{|\beta_r|}{2}} \leq k^{(\alpha-\frac{1}{2})(\gamma_r -|\beta_r|)} k^{-\frac{|\beta_r|}{2}} \leq C k^{(\alpha-\frac{1}{2})\gamma_r},
\]
 since $t_r \geq \gamma_r-|\beta_r| >0$ and $\alpha-\frac{1}{2}<0$. On the other hand, when $\gamma_r \leq |\beta_r|$, we still have $t_r \geq 0$ and so 
\[
k^{(\alpha-\frac{1}{2})t_r} k^{-\frac{|\beta_r|}{2}} \leq C k^{-\frac{|\beta_r|}{2}} \leq C k^{-\frac{\gamma_r}{2}} \le Ck^{(\alpha-\frac{1}{2})\gamma_r},
\]
since $\alpha > 0.$ Thus, there exists a $C_\beta>0$ so that
\begin{equation}\label{finalproductest1}
\left|\chi_{k,\alpha}(\xi) \p_{\xi}^{\beta_r} \lp \frac{\xi_r}{|\xi|}\rp^{\gamma_r} \right| \leq C_\beta k^{(\alpha-\frac{1}{2})\gamma_r}.
\end{equation}
An analogous argument shows
\begin{equation}\label{finalproductest2}
\left|\chi_{k,\alpha}(\xi) \p_{\xi}^{\beta_1} \lp\frac{\xi_1-|\xi|}{|\xi|}\rp^{\gamma_1} \right|\leq C_\beta k^{(\alpha-\frac{1}{2})\gamma_1},
\end{equation}
for some potentially different $C_\beta>0.$ Combining \eqref{finalproductest1} and \eqref{finalproductest2} with \eqref{chi_deriv}  and \eqref{multiindexexpand} yields
\[
\left| \p_{\xi}^{\beta} \lp \chi_{k,\alpha}(\xi)\lp\frac{\xi}{|\xi|} - \xi_0 \rp^{\gamma}\rp\right| \leq C_\beta k^{(\alpha-\frac{1}{2}) (\gamma_1+\gamma_2+\cdots+\gamma_n)} = C_\beta k^{(\alpha-\frac{1}{2})|\gamma|} = C_\beta k^{(\alpha-\frac{1}{2})\ell_j}.
\]
We have therefore proved \eqref{xi_factor}. 

Combining \eqref{Bj2_bound1} and \eqref{xi_factor}, we have that for any $\nu \ge 0$ and any $N$ large enough, there exists $C_\nu,C_\nu'>0$ so that
\[\left|\mathcal{A}_{j,2}(x)\right| \le C_{\nu} |x|^{-\nu}k^{\frac{m-j}{2} + (\alpha-\frac{1}{2})\ell_j}\!\!\!\sum\limits_{|\beta_1|\le\nu}\sum\limits_{|\gamma| = \ell_j}\!\!k^{\frac{m-j-|\beta_1|}{2}}\!\!\int\limits_{\R^n}\!(1+|\xi-k^{\frac{1}{2}}\xi_0|)^{-N}\!\!\,d\xi \le C_{\nu}' |x|^{-\nu} k^{m-j + (\alpha-\frac{1}{2})\ell_j}.\]
We then have
\begin{equation}\label{Bj2_bound2}
\left|\mathcal{A}_{j,2}(x)\right|\le C_{\nu}|x|^{-\nu}k^{m-j+(\alpha-\frac{1}{2})\ell_j },\quad\text{ for all }|x|\ge 1,
\end{equation}
for some $C_{\nu} > 0$. 

To estimate $\mathcal{A}_{j,2}(x)$ when $|x|\le 1$, we repeat the above argument without integrating by parts. From this, we obtain
\begin{equation}\label{Bj2_bound3}
\left|\mathcal{A}_{j,2}(x)\right| \le C k^{m-j +(\alpha-\frac{1}{2})\ell_j}, \quad \text{ for all }|x|\le 1.
\end{equation}
Choosing $\nu> \frac{n-1}{2}$, we can combine \eqref{Bj2_bound2} with \eqref{Bj2_bound3}, then take $L^2$ norms to obtain \eqref{A_2_bound} as desired.
\end{proof}

Recalling the constructions of $\mathcal{R}_j, \mathcal{A}_{j,1},$ and $\mathcal{A}_{j,2}$, we combine Lemmas \ref{R_j_lemma}, \ref{A_1_lemma}, and \ref{A_2_lemma} to obtain that for each $0 \le j \le N_0-1$,
\begin{equation}\label{Final_bound}
\|\Op(a_{m-j})h_k\|_{L^2(\R^n)} \le \|\mathcal A_{j,1}\|_{L^2(\R^n)} + \|\mathcal{A}_{j,2}\|_{L^2(\R^n)} +\|\mathcal{R}_{j}\|_{L^2(\R^n)} \le C_j k^{m-j+(\alpha-\frac{1}{2})\ell_j},
\end{equation}
for some $C_j > 0$. Since $\ell_j=\ell-2j$, \eqref{remainder_bound} and \eqref{Final_bound} imply that for any $N_0 \ge m$, there exists a collection of constants $\{C_j\}_{j=0}^{N_0}$ so that
\begin{align*}
\ltwo{\Op(a) h_k} &\leq \sum_{j=0}^{N_0-1} \ltwo{\Op(a_{m-j}) h_k} + \ltwo{\Op(r_{N_0}) h_k} \\
&\leq \sum_{j=0}^{N_0-1} C_j k^{m-j+(\alpha-\frac{1}{2})\ell_j} + C_{N_0} k^{\frac{m-N_0}{2}}\\
&= \sum_{j=0}^{N_0-1} C_j k^{m-\frac{\ell}{2}+(\ell-2j)\alpha}+C_{N_0}k^{\frac{m-N_0}{2}} \\
&  \leq Ck^{m-\frac{\ell}{2}+\ell\alpha} + C_{N_0}k^{\frac{m-N_0}{2}},
\end{align*}
for some $C>0.$ Choosing $N_0$ large enough and $\alpha$ small enough completes the proof of \propref{gaussian_beam_prop}. 
\end{proof}

\section{The upper bound for $\alpha$} \label{upper}
In this section, we show that $\alpha \le 2\min\{-D_0,L_\infty\}$, where $D_0$ and $L_\infty$ are defined as in \secref{Introduction}. That $-2D_0$ is an upper bound is straightforward to show. To do so, let $\lambda_j \in \text{Spec}(A_W) \backslash \{0\}$. Thus there exists $u=(u_0, u_1) \neq 0$ such that $A_W u = \lambda_j u$, where we recall 
\[A_W = \lp\begin{array}{cc}
0 & \Id\\
\Delta_g & -2W
\end{array}\rp.\] 
It is then immediate that $u(x) = e^{t\lambda_j}u_0(x)$ solves the damped wave equation with initial data $(u_0,u_1)$, and 
$$
E(u,t) = e^{2 t \Real(\lambda_j)} E(u,0).
$$
Since $E(u,0) \neq 0$, we have that $\alpha \le - 2\Real(\lambda_j)$ for all $j$. Furthermore, by the definition of $D_0$, there must either exist some $\lambda_{j_0}$ with $\Real(\lambda_{j_0}) = D_0$, or a sequence of $\lambda_j$ with $\Real(\lambda_j)\to D_0$. In either case, we must have $\alpha \leq -2D_0$.

Showing that $2L_\infty$ is also an upper bound is more complicated. Our technique for this is inspired by the method of Gaussian beams introduced by Ralston \cite{Ralston1969,Ralston1982}. Using Gaussian beams, one can produce quasimodes for the wave equation with energy strongly localized near a single geodesic. Intuitively, solutions to \eqref{damped_wave} should decay only when they interact with the damping coefficient. Motivated by this, we modify the Gaussian beam construction using the propagation of the defect measure derived in \secref{defectmeasure}, in analogy to \cite{Klein2017}. From this, we obtain solutions whose energy decays at a rate proportional to the integral of the symbol of $W$ along the chosen geodesic. 


To begin, we recall Ralston's original Gaussian beam construction on $\Rn$ with a Riemannian metric $g$. Let $A(t)$ be an $n\times n$ symmetric matrix-valued function with positive definite imaginary part. Let $t\mapsto (x_t,\xi_t)$ denote a geodesic trajectory and set 
$$
\psi(x,t)=\<\xi_t, x-x_t\> + \frac{1}{2} \<A(t)(x-x_t), x-x_t\>.
$$
Let $b\in C^\infty(\R\times\R^n)$. Then, we define 
\begin{equation} \label{uk_definition}
u_k(x,t) = k^{-1+n/4} b(t,x) e^{ i k \psi(x,t)}.
\end{equation}
The work of \cite{Ralston1982} guarantees that there exist appropriate choices of $b$ and $A(t)$ so that $u_k$ is a quasimode of the undamped wave equation with positive energy, which is concentrated along the geodesic $(x_t, \xi_t)$. We summarize some notable facts from \cite{Ralston1982} in the following Lemma. 

\begin{lemma}[\cite{Ralston1982}]\label{Ralston_lemma} Fix $T> 0$ and $(x_0,\xi_0) \in S^*M$. For $\vphi_t(x_0,\xi_0)=(x_t,\xi_t),$ there exists a $b\in C^\infty(\R\times\R^n)$ and an $n\times n$ symmetric matrix-valued function $t\mapsto A(t)$ so that for any $k\ge 1,$ for the $u_k$ defined in \eqref{uk_definition}
\begin{equation}\label{Rn_quasimode1}
\sup_{t \in [0,T]} \|\p_t^2   u_k(\cdot, t) - \Delta_{  g}   u_k(\cdot, t)\|_{L^2(\R^n)} \leq C k^{-\frac{1}{2}}.
\end{equation}
Furthermore, for all $t \in [0,T]$,
\begin{equation}\label{energy_limit}
\lim_{k \ra \infty} E(u_k,t) > 0,
\end{equation}
and the limit is always finite and independent of $t$. 
\end{lemma}

\begin{rmk}\textnormal{
  By \eqref{energy_limit}, we may assume without loss of generality that $\lim\limits_{k\to\infty}E(u_k,t) = 1$ for all $t\in[0,T].$ 
}
\end{rmk}

\begin{rmk}\label{chart_rmk}
\textnormal{
Using coordinate charts and a partition of unity, we can extend this construction to the case of manifolds, which results in a sequence $\{u_k\}\subset C^\infty(\R^+\times M)$ such that $\lim_{k \ra \infty} E(u_k, t)=1$ and the appropriate analogue of \eqref{Rn_quasimode1} holds. 
}
\end{rmk}
Next, we modify $\{u_k\}$ using the propagation of the defect measure from \secref{defectmeasure} to produce a sequence of quasimodes for the damped wave equation. Recall that $G_t^+: S^* M \ra \Cb$ is given by $G_t^+(x_0,\xi_0) = \exp\left( - \int_0^t 2w(x_s,\xi_s) ds \right)$ where $(x_s,\xi_s)=\vphi_s(x_0,\xi_0)$.

Recall also the time averaging function $t\mapsto L(t)$ 
\[L(t) = \frac{1}{t}\inf\limits_{(x_0,\xi_0)\in S^*M}\int\limits_0^t w(x_s,\xi_s)\,ds.\]
Note that $L(t)$ can be rewritten in terms of $G_t^+$ 
\[L(t) = -\frac{1}{t}\sup\limits_{(x,\xi)\in S^*M}\ln\lp G_t^+(x,\xi)\rp.\]
Motivated by the form of $G_t^+$, we fix $(x_0,\xi_0)\in S^*M$ and set 
\[v_k(t,x) = G_t^+(x_0,\xi_0) u_k(t,x).\]
That is, we modify the quasimode for the free wave equation so that it decays exponentially at a rate proportional to integral of $w(x_t,\xi_t)$ along the geodesic it is concentrated on. We now show that for any $\ve > 0,$ $v_k$ is an $\mathcal O(k^{-\frac{1}{2} + \ve})$ quasimode of \eqref{damped_wave}.

\begin{prop}\label{dw_quasimode}
Given $(x_0,\xi_0) \in S^* M$, let $u_k(t,x)$ be as specified in Remark \ref{chart_rmk} and set $v_k(t,x) = G_t^+(x_0,\xi_0)u_k(t,x)$. For any $T > 0$ and $\ve > 0$, there exists a constant $C_{\ve,T} > 0$ so that
\begin{equation}\label{dw_quasimode_eqn}
\sup\limits_{t\in[0,T]}\|(\partial_t^2 - \Delta_g + 2W\partial_t)v_k(t,\cdot)\|_{L^2(M)}\le C_{\ve,T} k^{-\frac{1}{2} + \ve}.
\end{equation}
\end{prop}
\begin{proof}
By direct computation
\begin{align*}
(\p_t^2 -\Delta_g + 2W \p_t) v_k &=  G_t^+ (\p_t^2 - \Delta_g) u_k + 2 \p_t  G_t^+ \p_t u_k + (\p_t^2  G_t^+)u_k + 2 W \p_t ( G_t^+ u_k) \\
&=  G_t^+ (\partial_t^2 -\Delta_g) u_k -2 w(x_t,\xi_t)  G_t^+ \p_t u_k - \p_t(w(x_t,\xi_t)  G_t^+)u_k \\
& \hskip 0.5in + 2 W  G_t^+ \p_t u_k -2w(x_t, \xi_t) W   G_t^+ u_k \\
&=  G_t^+ (\partial_t^2 -\Delta_g) u_k + 2 (W-w(x_t, \xi_t))  G_t^+ \p_t u_k \\
& \hskip 0.5in+ \lp w(x_t, \xi_t)^2 - 2w(x_t,\xi_t)W  - \p_t w(x_t, \xi_t) \rp  G_t^+ u_k.
\end{align*}
 By the construction of $u_k$ and the boundedness of $G_t^+$, we have
 \[\sup\limits_{t\in[0,T]}\|G_t^+ (\partial_t^2 -\Delta_g) u (t,\cdot)\|_{L^2(M)} \leq O(k^{-\frac{1}{2}}).\]
 Since $W$ is order zero, and therefore bounded on $L^2(M)$, we obtain
$$
\sup\limits_{t\in[0,T]}\ltwoo{(w(x_t, \xi_t)^2 - 2w(x_t,\xi_t) W  - \p_t w(x_t,\xi_t)) G_t^+ u_k}{M} \leq C \sup\limits_{t\in[0,T]}\ltwo{u_k(t,\cdot)} =\mathcal O(k^{-1}),
$$ 
where the final equality follows from the fact that $\int\limits_{\R^n}k^{\frac{n}{4}}e^{-k|y|^2}\,dy$ is uniformly bounded in $k$. 

To estimate $W-w(x_t,\xi_t)) G_t^+ \p_t u_k(t, \cdot)$ we will apply \propref{gaussian_beam_prop} with $m=0$ and $\ell=1$. To do so, note 
$W-w(x_t,\xi_t)$ is a pseudo with appropriate vanishing properties. Furthermore
 \[\partial_t u_k(x,t) = k^{-1+\frac{n}{4}}\partial_t b(t,x) e^{ik\psi(x,t)} + ik^{\frac{n}{4}}b(t,x)\partial_t\psi(x,t) e^{ik\psi(x,t)},\]
and for fixed $t$, both of these terms take the form of a coherent state $h_k$ as defined in \propref{gaussian_beam_prop} (the fact that the first has an extra factor of $k^{-1}$ is irrelevant, as it only improves the estimate). Since all quantities depend on $t$ in a $C^\infty$ fashion, for any $\ve>0$
\begin{equation}\label{W_G_eqn}
\sup\limits_{t\in[0,T]}\ltwoo{ 2 (W-w(x_t, \xi_t)) G_t^+ \p_t u_k(t,\cdot)}{M} \leq C(k^{-1/2 + \ve}).
\end{equation}
By the triangle inequality, we obtain \eqref{dw_quasimode_eqn}, which completes the proof.
\end{proof}

The next step in the proof of the upper bound for $\alpha$ is given a $(x_0,\xi_0)$  produce a sequence of \emph{exact} solutions to \eqref{damped_wave} whose energy approaches $|G_t^+(x_0,\xi_0)|^2$.

\begin{prop}\label{exact_solution_prop}
Given any $T > 0$, any $\ve > 0$, and any $(x_0,\xi_0)\in S^*M$, there exists an exact solution $u$ of the generalized damped wave equation \eqref{damped_wave} with \begin{equation*}
\left|E(u,0) - 1\right| < \ve 
\end{equation*}
and
\begin{equation}\label{omega_eqn}
\left|E(u,T) - |G_T^+(x_0,\xi_0)|^2\right| < \ve.
\end{equation} 
\end{prop}

\begin{proof}
Let $u_k$ and $v_k$ be as defined previously. Then, define $\omega_k$ as the unique solution of the damped wave equation with initial conditions $\omega_k(x,0) = v_k(x,0)$ and $\p_t \omega_k(x,0) = \p_t v_k(x,0)$. 
It is immediate that 
\[E(\omega_k,0) = E(v_k,0) = E(u_k,0)\to 1,\quad \text{as }k\to\infty.\]

 To see \eqref{omega_eqn}, first note by the triangle inequality
\begin{equation}\label{omegaktriangle}
|E(\omega_k, t)^{\frac{1}{2}} - E(v_k,t)^{\frac{1}{2}}| \leq E(\omega_k-v_k, t)^{\frac{1}{2}}.
\end{equation}
Thus, it suffices to prove that $\lim_{k \ra \infty} E(v_k,t) =|G_t^+(x_0,\xi_0)|^2$ and that $\lim\limits_{k\to\infty}E(\omega_k-v_k,t) =0$. To see that $\lim_{k \ra \infty} E(v_k, t) = |G_t^+(x_0,\xi_0)|^2$, note that by the definition of $v_k$ and properties of $G_t^+$
\begin{align*}
E(v_k,t) & = \frac{1}{2} \int\limits_M |G_t^+(x_0, \xi_0) \p_t u_k(x, t) - w(x_t, \xi_t) G_t^+(x_0, \xi_0) u_k(x, t)|^2\\
& \hskip 0.7in + |G_t^+(x_0,\xi_0) \nabla_g u_k(x, t)|^2\, dv_g(x). 
\end{align*}
Now since $w(x_t,\xi_t)$ and $G_t^+$ are bounded
\[\int\limits_M\left|w(x_t, \xi_t) G_t^+(x_0, \xi_0) u_k(x, t)\right|^2\,dv_g(x) \le C\ltwo{u_k(t,\cdot)}^2 \leq C'k^{-2} ,\]
for some $C,\,C' > 0$. Thus, 
\begin{align}\label{quasienergyapproach}
\lim\limits_{k \ra \infty} E(v_k, t) & = \lim\limits_{k\to\infty}\frac{1}{2}\int\limits_{M}\left| G_t^+(x_0,\xi_0)\partial_t u_k(x,t)\right|^2  + \left|G_t^+(x_0,\xi_0)\nabla_g u_k(x,t)\right|^2\,dv_g(x)\nonumber\\
& = \left|G_t^+(x_0,\xi_0)\right|^2\lim\limits_{k\to\infty}E(u_k,t) \nonumber\\ 
& = \left|G_t^+(x_0,\xi_0)\right|^2,
\end{align}
where in the final equality we used that $\lim\limits_{k\to\infty}E(u_k,t) = 1$.

To control $E(\omega_k-v_k, t)$, let $f_k = (\p_t^2 -\Delta + 2 W \p_t) v_k.$ Then 
\[(\p_t^2 - \Delta +2 W \p_t)(v_k-\omega_k)=f_k.\]
By \propref{dw_quasimode}, for any $\ve,T > 0$ there exists a $C_{\ve,T} > 0$ such that 
\begin{equation}\label{fkcontrol}
\sup\limits_{t\in[0,T]}\|f_k(t,\cdot)\|_{L^2(M)}\le C_{\ve,T} k^{-\frac{1}{2} + \ve}.
\end{equation}
By direct computation
\begin{align*}
\p_t E(\omega_k-v_k, t) & = \int\limits_{M}(\partial_t^2 - \Delta_g)(\omega_k-v_k)\partial_t\overline{(\omega_k-v_k)} +(\partial_t^2 - \Delta_g)\overline{(\omega_k-v_k)}\partial_t(\omega_k-v_k)\,dv_g(x)\\
&= 2\Real\int\limits_M [f_k- 2W\partial_t(\omega_k-v_k)]\partial_t\overline{(\omega_k-v_k)}\,dv_g(x)\\
& = 2\Real\int\limits_{M} f_k\cdot \partial_t \overline{(\omega_k-v_k)}\,dv_g(x)-4\Real\langle W\partial_t(\omega_k-v_k),\partial_t(\omega_k-v_k)\rangle_{L^2(M)} .
\end{align*}
Note that the second term on the right-hand side above is nonpositive, since $W$ is a nonnegative operator. Now, using \eqref{fkcontrol} and that $\ltwo{\p_t (v_k-\omega_k)(t,\cdot)}$ is uniformly bounded for $k \in \mathbb{N}$ and $t \in [0,T]$, there exists $C_{\ve,T}' > 0$ such that
\[\sup\limits_{t\in[0,T]}\left|2\Real\int\limits_M f_k\partial_t\overline{(\omega_k-v_k)}\,dv_g(x)\right| \le 2\|f_k(t,\cdot)\|_{L^2}\|\partial_t(\omega_k-v_k)(t,\cdot)\|_{L^2}\le C_{\ve,T}'k^{-\frac{1}{2} + \ve}.\]
Thus, for any $\ve > 0$
\[\sup\limits_{t\in[0,T]}\left|\p_t E(\omega_k-v_k,t) \right|\leq C_{\ve,T}'k^{-\frac{1}{2}+\ve}.\] 
Integrating in $t$ gives 
$$
\sup_{t \in [0,T]} E(v_k-\omega_k,t) \leq C_{\ve,T}'T k^{-\frac{1}{2}+\ve}.
$$

Combining this with \eqref{omegaktriangle} and \eqref{quasienergyapproach} yields \eqref{omega_eqn}.
\end{proof}

For the penultimate step in the proof of the upper bound for $\alpha$, we will show that $t \mapsto t L(t)$ is superadditive. That is, for $r,t\geq 0$, $(t+r) L(t+r) \geq t L(t) + rL(r)$. To see this observe
\begin{align*}
(t+r) L(t+r) &= \inf_{(x_0,\xi_0) \in S^*M} \int_0^{r+t} w(x_s, \xi_s) ds \\
&=\inf_{(x_0,\xi_0) \in S^*M} \left(\int_0^{t} w(x_s, \xi_s) ds+\int_t^{t+r} w(x_s,\xi_s) ds \right) \\
& \geq \inf_{(x_0,\xi_0) \in S^*M} \int_0^t w(x_s,\xi_s) ds + \inf_{(x_0,\xi_0) \in S^* M} \int_t^{t+r} w(x_s,\xi_s) ds \\
&= \inf_{(x_0,\xi_0) \in S^* M} \int_0^t w(x_s, \xi_s) ds + \inf_{(x_0,\xi_0)\in S^*M} \int_0^r w(x_s,\xi_s) ds\\
& = tL(t)+ rL(r).
\end{align*}
Now by Fekete's lemma, $L_\infty :=\lim\limits_{t\to\infty}L(t) = \sup\limits_{t\in [0,\infty)}L(t)$, and thus $L(t) \le L_\infty$ for all $t$. That the supremum is not infinite follows from the fact that $w(x,\xi)$ is uniformly bounded on $T^*M.$ 

We are now ready to show that $\alpha \leq 2 L_{\infty}$. Assume for the sake of contradiction that $\alpha = 2 L_{\infty} + 3\eta$ for some $\eta>0$. Then since $2(L_\infty + \eta) < \alpha$, there exists a $C > 0$ such that for all $t \geq 0$ and all solutions $u$ of \eqref{damped_wave}, 
\begin{equation}\label{contra_estimate}
E(u,t) \leq C E(u,0) e^{-2t(L_{\infty} + \eta)}.
\end{equation}
For the next step, it is convenient to remove the factor of $C.$ To accomplish this,  choose $T>0$ large enough so that $\max(C,1) <e^{T \eta}$. Then
$$
C e^{-2T(L_{\infty} + \eta)} < e^{-T(2L_{\infty}+\eta)}.
$$
Since $L(t)\le L_\infty$ for all $t$, we obtain 
\begin{equation}\label{contradecayequation}
Ce^{-2T(L_\infty + \eta)} < e^{-2TL_{\infty}-T\eta} \leq e^{-2TL(T)-T\eta}.
\end{equation}
Now, we recall that
\[-TL(T) = \sup\limits_{(x,\xi)\in S^*M}\ln G_T^+(x,\xi).\]
Thus, there exists a point $(x_0,\xi_0)\in S^*M$ such that $\ln G_T^+(x_0,\xi_0) > -T L(T) -\frac{1}{2}T\eta.$ Therefore,
\[e^{-2TL(T) - T\eta} <  |G_T^+(x_0,\xi_0)|^2.\]
So by \eqref{contradecayequation} there exists a $\delta > 0$ such that 
\[Ce^{-2T(L_\infty + \eta)} < |G_T^+(x_0,\xi_0)|^2 - \delta.\]
Now, by \propref{exact_solution_prop}, there exists an exact solution $u$ of \eqref{damped_wave} such that
\[1 > E(u,0) - \frac{\delta}{2} \quad \text{ and }\quad E(u,T) > |G_T^+(x_0,\xi_0)|^2 - \frac{\delta}{2}.\]
Thus,
\begin{align*}
E(u,T) & > E(u,T)\lp E(u,0) - \frac{\delta}{2}\rp\\
& > E(u,0)\lp |G_T^+(x_0,\xi_0)|^2 -\frac{\delta}{2}\rp - \frac{\delta}{2}E(u,T)\\
& > E(u,0)\lp |G_T^+(x_0,\xi_0)|^2 -\frac{\delta}{2}\rp - \frac{\delta}{2}E(u,0)\\
& = E(u,0)\lp |G_T^+(x_0,\xi_0)|^2 -\delta\rp.
\end{align*}
Therefore,
\[E(u,T) > E(u,0)(|G_T^+(x_0,\xi_0)|^2 - \delta) > C E(u,0)e^{-2T(L_\infty + \eta)},\]
but this contradicts \eqref{contra_estimate}. Thus, we must have $\alpha \le 2L_\infty.$ Combining this with the discussion at the beginning of this section, we have proved the upper bound 
\[\alpha \le 2\min\{-D_0,L_\infty\}.\]
We complete the proof of \thmref{best_const_thm} in the next section by proving the corresponding lower bound for $\alpha.$ 

\section{The lower bound for $\alpha$}\label{lower}
In this section, we prove that the best exponential decay rate satisfies 
\begin{equation}\label{lower_bound}
\alpha \ge 2\min\{-D_0,L_\infty\},
\end{equation}
which is the final component of the proof of \thmref{best_const_thm}. In contrast to the proof of the upper bound, this section proceeds in direct analogy to the work of Lebeau, and so we omit many of the details which can be found in \cite{Lebeau1996,Klein2017}. While the proofs presented here are not new, we include them to introduce notation that is used later in \secref{thm_1_proof}, where we use \thmref{best_const_thm} to prove \thmref{exp_decay_thm}. 


We begin with the following energy inequality, which for the multiplicative case is presented as Lemma 1 in \cite{Lebeau1996}.
\begin{lemma}\label{energy_lemma}
For every $T> 0$ and every $\ve > 0$, there exists a constant $c(\ve,T) > 0$ so that for every solution $u$ of \eqref{damped_wave}, 
\begin{equation}
E(u,T) \le (1+\ve)e^{-2TL(T)}E(u,0) + c(\ve,T)\|(u_0,u_1)\|_{L^2\bigoplus H^{-1}}^2,
\end{equation}
\end{lemma}

\noindent This inequality is proved using straightforward properties of the propagation of the defect measure, so the proof from \cite{Lebeau1996} goes through with no modification. To obtain the desired lower bound on $\alpha$ we must further control the $\|(u_0,u_1)\|_{L^2\bigoplus H^{-1}}^2$ on the right hand side. 

Given \lemref{energy_lemma}, we proceed by introducing the adjoint $A_W^* = \lp\begin{array}{cc}0 & -\Id\\ -\Delta_g& -2W\end{array}\rp$ of the semigroup generator $A_W$. Note that the spectrum of $A_W^*$ is the conjugate of the spectrum of $A_W$. Thus, we denote by $E_{\lambda_j}^*$ the generalized eigenspace of $A_W^*$ with associated eigenvalue $\overline{\lambda_j}$. Recall that $\mathscr H = H^1(M)\oplus L^2(M)$, equipped with the natural norm. It is also useful to introduce the $\dot {\mathscr H}$ seminorm defined for elements of $\mathscr H$ by 
\[\|(u_0,u_1)^T\|_{\dot{\mathscr H}}^2 = \|\nabla u_0\|_{L^2}^2 + \|u_1\|_{L^2}^2.\]
For each $N \ge 1$, define the subspace 
\[H_N = \left\{ \varphi\in \mathscr H:\, \langle\varphi,\psi\rangle_{\mathscr H} = 0, \, \forall \psi\in \bigoplus\limits_{|\lambda_j|\le N} E_{\lambda_j}^*\right\}.\]
Our first observation is that $H_N$ is invariant under the action of the semigroup $e^{tA_W}$. To demonstrate this, let $\{\psi_k\}$ be a basis of the finite dimensional space $\bigoplus\limits_{|\lambda_j|\le N}\!\! E_{\lambda_j}^*\subset D(A_W^*).$ Now,  since $E_{\lambda_j}^*$ is invariant under $A^*_W$, we can express each $A_W^* \psi_l$ as a finite linear combination of the $\{\psi_k\}$. Thus, for each $\ell$ and any $\varphi\in H_N$, we have
\[
\partial_t\langle e^{tA_W}\varphi,\psi_\ell\rangle_{\mathscr H}\big|_{t= 0} = \langle e^{tA_W}\varphi,A_W^*\psi_\ell\rangle_{\mathscr H}\big|_{t= 0} = \sum c_{\ell,k}\langle \varphi,\psi_k\rangle_{\mathscr H} = 0,
\]
 by the definition of $H_N$. Repeating this argument, we see that $\partial_t^j\langle e^{tA_W}\varphi,\psi_\ell\rangle_{\mathscr H}\big|_{t= 0} = 0$ for all $j.$ Observing that $\langle e^{tA_W}\varphi,\psi_\ell\rangle_{\mathscr H}$ is an analytic function of $t$, we have $\langle e^{tA_W}\varphi,\psi_\ell\rangle_{\mathscr H} = 0$ for all $t\in\R.$ Therefore $e^{tA_W}\varphi\in H_N$. 

Now, define $\mathscr H' = L^2\oplus H^{-1}$ and let $\theta_N$ denote the norm of the embedding of $H_N$ in $\mathscr H'$, which is well-defined since $M$ is compact. Since $W$ is bounded on $L^2$, it is compact as an operator from $L^2\to H^{-1}$. Therefore $A_W^*:\mathscr H\to \mathscr H'$ is a compact perturbation of the skew-adjoint operator $\lp\begin{array}{cc}0 & - \Id\\ -P & 0\end{array}\rp$. Thus, the family $\{E_{\lambda_j}^*\}_{j=0}^\infty$ is total in $\mathscr H$, and so $\lim_{N \ra \infty} \theta_N = 0$ (c.f. \cite[Ch. 5, Theorem 10.1]{GKBook1969}).

We can now proceed with the proof of \eqref{lower_bound}. Assume that $2\min\{-D_0,L_\infty\} > 0$, otherwise the statement is trivial. Choose $\eta > 0$ small enough so that $\beta = 2\min\{-D_0,L_\infty\} - \eta > 0$ and take $T$ large enough so that $4|L_\infty - L(T)| < \eta$ and $e^{\frac{\eta T}{2}}>3$. Then, by \lemref{energy_lemma} with $\ve=1$, there exists a constant $c(1,T)$ such that for every solution $u$ of \eqref{damped_wave}
\begin{equation}\label{energylowboundintermed}
E(u,t) \leq 2 e^{-2TL(T)}E(u,0) + c(1,T)\|(u_0,u_1)\|_{\mathscr H'}^2.
\end{equation}
Next, choose $N$ large enough so that $c(1,T)\theta_N^2 \le e^{-2TL(T)}$. Then, for solutions $u$ of \eqref{damped_wave} with initial data $(u_0,u_1)^T\in H_N$ 
\[E(u,T)\le 3 e^{-2TL(T)}E(u,0).\]
Since $H_N$ is invariant under evolution by $e^{tA_W}$
\[E(u,kT) \le 3^k e^{-2k TL(T)}E(u,0), \quad\forall k\in\N.\]
Then, we can use the fact that $4|L_\infty-L(T)|< \eta$ and $\frac{\eta T}{2} > \ln 3$ to obtain that
\begin{align*}
E(u,kT) & \le 3^k e^{-2kT(L_\infty - \eta/4)}E(u,0)\\
& \le \lp e^{\ln 3 - \frac{\eta T}{2}}\rp^k e^{-2kTL_\infty}E(u,0)\\
& \le e^{-kT\beta}E(u,0),
\end{align*}
where the final inequality follows from the fact that $\beta \le 2L_\infty -\eta < 2L_\infty$ by definition. Since the energy is nondecreasing, it follows that 
\begin{equation}\label{H_N_energy}
E(u,t)\le C e^{-\beta t}E(u,0)\, \quad \forall t\ge 0,
\end{equation}
for some constant $C > 0.$ 

To extend \eqref{H_N_energy} to all solutions of \eqref{damped_wave}, let $\Pi$ denote the orthogonal projection from $\mathscr H$ onto $\bigoplus\limits_{|\lambda_j|\le N}\!\! E_{\lambda_j}.$ Then for any $v = (u_0,u_1)^T\in \mathscr H$, there is an orthogonal decomposition of the form $v = \Pi v + (\Id - \Pi)v$. Since $E_{\lambda_j}$ and $E_{\lambda_k}^*$ are orthogonal for $\lambda_j\ne \lambda_k$, we have that $(\Id - \Pi)v\in H_N,$ and hence $H_N^\perp = \bigoplus\limits_{|\lambda_j|\le N}\!\! E_{\lambda_j}$. Since $E_{\lambda_j}$ is invariant under $e^{tA_W}$ and $H_{N}^\perp$ is finite dimensional, we have that there exists a $C>0$ so that for all solutions $u$ of \eqref{damped_wave} with initial data in $H_N^\perp$,
\begin{equation}\label{W_N_energy}
E(u,t)\le Ce^{2D_0}E(u,0)\le Ce^{-\beta t}E(u,0),\quad \forall t\ge 0.
\end{equation}
Finally, since $\Pi$ and $\Id - \Pi$ are continuous with respect to the $\dot {\mathscr H}$ seminorm, for some $C >0$
\[E(\Pi u,0) + E((\Id - \Pi)u,0) \le C E(u,0).\]
Therefore, using the decomposition $\Pi+(\Id-\Pi)$ on the initial data of any solution $u$ we can apply \eqref{H_N_energy} and \eqref{W_N_energy} to obtain 
\begin{equation}
E(u,t)\le C e^{-\beta t}E(u,0),\quad\forall t\ge 0,
\end{equation}
for some possibly larger $C>0.$ By definition of the best possible decay rate, $\alpha \ge \beta = 2\min\{-D_0,L_\infty\} - \eta$. Since $\eta$ can be taken arbitrarily small, this proves \eqref{lower_bound}. Combining this with the upper bound obtained in Section \ref{upper} completes the proof of \thmref{best_const_thm}.

\section{Proof of \thmref{exp_decay_thm}}\label{thm_1_proof}
In this section we show that \thmref{best_const_thm} implies \thmref{exp_decay_thm}. First, we will assume both \assumsref{geometric_control}{unique_cont} are satisfied. We will show this implies $\alpha > 0$, which is equivalent to exponential energy decay. Note that \assumref{geometric_control} immediately implies that $L_\infty \ge c > 0$. Thus, we only need to show that $D_0 < 0$. For this, we introduce the quantity 
\[D_\infty :=\lim\limits_{R\to\infty}\sup\{\Real(\lambda):\,|\lambda| > R,\,\lambda\in\Spec{A_W}\}.\]
We claim that $D_\infty \le -L_\infty$. To show this, first recall $E_{\lambda_j}$ and $H_N$ from \secref{lower}. Let $u$ be a solution to \eqref{damped_wave} with initial data $(u_0,u_1)^T\in E_{\lambda_j}$ with $|\lambda_j| > N$.  Then $u = e^{tA_W}(u_0,u_1)^T = e^{t\lambda_j}(u_0,u_1)^T$. Note that $E_{\lambda_j}\subset H_N$ whenever $|\lambda_j| > N$. Combining this with the proof of \eqref{H_N_energy}
\[e^{2\Real(\lambda_j) t}E(u,0) = E(u,t)  \le Ce^{-\beta t}E(u,0), \]
for every $0 < \beta < 2L_\infty$. Hence, $2\Real(\lambda_j)\le -\beta$ whenever $|\lambda_j| \ge N$, and so $\Real(\lambda_j) \le -L_{\infty}$ for such $\lambda_j$. It immediately follows that $D_\infty  \le -L_\infty < 0$. 

By the abstract spectral theory arguments in the proof of \cite[Lemma 4.2]{AnantharamanLeautaud2014}, the spectrum of $A_W$ consists only of isolated eigenvalues and $\Real(\lambda)\le 0$ for all $\lambda\in\Spec(A_W)$. Thus, in order to have $D_0 = 0$, either $D_\infty = 0$ or there exists a nonzero eigenvalue of $A_W$ on the imaginary axis. Since we have already shown $D_\infty < 0$, we need only rule out nonzero imaginary eigenvalues. Suppose $i\lambda\in \text{Spec}(A_W)$ with $\lambda \in\R$ and corresponding eigenvector $(v_0,v_1)^T.$ Then $v_1 = \lambda v_0$, and
\begin{equation}\label{stationaryeq}
\Delta_g v_0 +\lambda^2 v_0 - 2i\lambda W v_0 = 0.
\end{equation}
Taking the $L^2$ inner product of both sides with $v_0$ and then taking the imaginary part gives
\[- 2\lambda\langle Wv_0,v_0\rangle = 0.\]
If $\lambda = 0$, the equation is trivially satisfied. However, if $\lambda \ne 0$, then $\langle W v_0,v_0\rangle = 0$. Recalling that $W = \sum B_j^*B_j$ for some collection of operators $B_j$, we must have $W v_0 = 0$. Then by \eqref{stationaryeq} $v_0$ is an eigenfunction of $\Delta_g$ with eigenvalue $-\lambda^2$ and $v_0 \in \ker W$. But by \assumref{unique_cont}, this is impossible. Thus, the only possible eigenvalue of $A_W$ on the imaginary axis is zero and we cannot have $D_0 = 0.$ Combining this with the fact that $L_\infty > 0$, we have shown that \assumsref{geometric_control}{unique_cont} imply $\alpha > 0$, which in turn demonstrates that solutions to \eqref{damped_wave} experience exponential energy decay. 


We now prove the reverse implication in \thmref{exp_decay_thm}. For this, we assume that \eqref{exp_decay_eqn} holds with some $\beta >0$ for all solutions $u$ and we want to see that \assumsref{geometric_control}{unique_cont} hold. By definition, $\alpha \ge \beta > 0$, and hence both $-D_0$ and $L_\infty$ are strictly positive. Because $L_\infty \ge \alpha/2 > 0$ \assumref{geometric_control} holds. Similarly, since $D_0 < 0$, there cannot be any eigenvalues of $A_W$ on the imaginary axis except possibly at zero. Now suppose that $v\in L^2$ satisfies $-\Delta_gv  =\lambda^2 v$ with $\lambda \ne 0$ and $Wv = 0.$ Then $(v,i\lambda v)^T$ is an eigenvector of $A_W$ with eigenvalue $i\lambda\ne 0$, which is a contradiction. Thus, \assumref{unique_cont} must also hold, which completes the proof of \thmref{exp_decay_thm}.

\section{A Class of Examples on Analytic Manifolds}
\label{examples}
\noindent One of the key hypotheses of \thmref{exp_decay_thm} was that the damping coefficient $W$ must not annihilate any eigenfunctions of $\Delta_g$ associated with nonzero eigenvalues. In the case where $W$ is a multiplication operator whose support satisfies the classical geometric control condition, this is always satisfied by the unique continuation properties of elliptic operators \cite{RauchTaylor1975}. However, when the damping is pseudodifferential it is much more difficult to check this hypothesis. 

In this section, we produce a collection of operators on real analytic manifolds which satisfy \assumref{unique_cont} and are not multiplication operators. We also give an example of an explicit pseudodifferential damping coefficient on $\mathbb T^2$ which satisfies \assumsref{geometric_control}{unique_cont}. The primary tool in this discussion is the analytic wavefront set, and so we begin by providing some background definitions for the reader's convenience. More details can be found in \cite[\S 8.4-8.6]{HormanderBook1983}.

Given a set $X\subseteq \R^n$ and a distribution $u\in \Dc'(X)$, if $u$ is real analytic on an open neighborhood of $x_0$ we write that $u\in C^a$ near $x_0\in X.$  In analogy with the relationship between the standard wavefront set and $C^\infty$ singularities, one can resolve $C^a$ singularities by defining the \textit{analytic} wavefront set, written $WF_A(u)$ and defined as follows. 

\begin{definition}
We say that a point $(x_0,\xi_0)\in T^*X\setminus0$ is not in $WF_A(u)$, if there exists an open neighborhood $U$ of $x_0$, a conic neighborhood $\Gamma$ of $\xi_0$ and a bounded sequence $u_N \in \Ec'(X)$, which are equal to $u$ on $U$, and each satisfy
\begin{equation}\label{localanalytic}
|\wh{u}_N(\xi)| \leq C \left(\frac{N+1}{|\xi|}\right)^N,
\end{equation}
for all $\xi\in\Gamma.$
\end{definition}
\noindent By \cite[Prop. 8.4.2]{HormanderBook1983}, we have that $u\in C^a$ near $x_0$ if and only if $WF_A(u)$ contains no points of the form $(x_0,\xi)$ with $\xi \ne0.$

We also introduce a set, which we can be thought of as the analytically invertible directions of $u$ denoted by $\Gamma_A(u)$. Its complement is commonly called the (analytic) characteristic set of $u$ \cite{HormanderBook1983}.
\begin{definition}
We say that $\xi_0\in \R^n\setminus0$ is in $\Gamma_A(u)$ if there exists a complex conic neighborhood $V$ of $\xi_0$ and a function $\Phi$, which is holomorphic in $\{\xi\in V:\,|\xi| > c\}$ for some $c > 0,$ satisfying $\Phi \wh u = 1$ in $V\cap \R^n$ and there exists $C,N >0$ such that
\[|\Phi(\zeta)|\le C|\zeta|^N,\]
for $\zeta\in V.$ 
\end{definition}
The final preliminary we require is the notion of the normal set of a closed region $F$ contained within a manifold $M$. For the purposes of this definition, we only require that $M$ be $C^2$. 
\begin{definition}
Let $F$ be a closed region in a $C^2$ manifold $M.$ The exterior normal set, $N_e(F)$, is defined as the set of all $(x_0, \xi_0)\in T^*M \setminus 0$ such that $x_0 \in F$ and such that there exists a real valued function $f \in C^2(M)$ with $df(x_0) = \xi_0 \neq 0$ and 
\[f(x) \leq f(x_0), \quad x \in F.\]
The interior normal set of $F$ is then defined by $N_i(F)=\{(x,\xi):\, (x, -\xi) \in N_e(F)\}$ and the full normal set is defined as $N(F)= N_e(F)\bigcup N_i(F)$. We write $\overline{N}(F)$ to denote the closure of the normal set of $F$.
\end{definition}
\noindent Note that the projection of $N_e(F)$ onto $M$ is dense in $\p F$ but might not be equal to $\p F$ \cite[Prop. 8.5.8]{HormanderBook1983}.

With these definitions in hand, we are able to describe a class of pseudodifferential operators which do not annihilate any eigenfunctions of $\Delta_g$.

\begin{lemma}\label{unique_cont_condition}
Let $(M,g)$ be a compact, real analytic manifold of dimension $n$. Suppose $\spacecutoff,\wt\spacecutoff\in C_c^\infty(M)$ are cutoff functions supported entirely within a single coordinate patch, with $\wt \spacecutoff \equiv 1$ on an open neighborhood of the support of $\spacecutoff$. Let $\freqcutoff\in C^\infty(\R^n)$ be homogeneous of degree 0 outside a compact neighborhood of the origin, and define $B\in\Psi_{cl}^0(M)$ in local coordinates by $Bu = \wt\spacecutoff\Op(\freqcutoff(\xi))\spacecutoff u$.  Let $\widecheck{\freqcutoff}$ denote the inverse Fourier transform of $\freqcutoff$ and $\pi_2:\R^n\times\R^n\to\R^n$ denote the natural projection onto the fiber variables $\xi$, if 
\begin{align*}
\pi_2(\overline N(\supp\spacecutoff)) \cap& \Gamma_A(\widecheck\freqcutoff) \ne \emptyset, 
\end{align*}
then for any eigenfunction $u$ of $\Delta_g$, we have $B u\ne 0.$

\end{lemma}

\begin{proof}
We proceed by contradiction, so assume $B u =0$ for some eigenfunction $u$ of $\Delta_g$. Thus $WF_A(B u)= \emptyset$ and we aim to show there exists some $(x_0, \xi_0) \in WF_A(Bu)$. First, by \cite[Thm 8.5.6']{HormanderBook1983}, we have 
\[\overline N(\supp\spacecutoff u)\subseteq WF_A(\spacecutoff u).\]
Since $ u$ is an eigenfunction, it cannot vanish identically on any open set. We claim that this implies
\begin{equation}\label{boundary_eqn}
\partial(\supp\spacecutoff)\subseteq \partial(\supp \spacecutoff u) .
\end{equation}
To see this, suppose $x\in\partial(\supp\spacecutoff)$ and let $V$ be any open neighborhood of $x$. Since $\spacecutoff(x) = 0$, we have that $\spacecutoff(x) u(x) = 0$, so it is enough to show that $\spacecutoff u$ is not identically zero on all of $V$. Without loss of generality, we may assume that $V$ lies entirely within the same coordinate patch containing $\supp\spacecutoff$. Since $x$ is a boundary point of the support, $\spacecutoff$ does not vanish identically on $V$. By continuity, this implies the existence of a smaller open neighborhood $\wt V\subset V$ (not containing $x$) where $\spacecutoff$ is never zero. Since $ u$ is an eigenfunction, it cannot vanish identically on $\wt V$, and hence $\spacecutoff  u$ is not identically zero on $\wt V\subseteq V$, which proves \eqref{boundary_eqn}.

Now we want to show $N(\supp \spacecutoff) \subset N(\supp \spacecutoff u)$. Take $(x_0, \xi_0) \in N(\supp \spacecutoff),$ note $x_0$ maximizes a function $f$ on $\supp \chi$ with $df(x_0)\neq0$, so $x_0 \in \p (\supp \spacecutoff) \subseteq \p (\supp \spacecutoff u)$. That is $x_0$ is not an interior point. Furthermore, since $\supp \spacecutoff \supseteq \supp\spacecutoff u$ and $f$ is maximized at $x_0$ in $\supp \spacecutoff$ it must also be maximized at $x_0$ when restricted to the smaller set $\supp \spacecutoff u.$ Therefore $N(\supp\spacecutoff)\subseteq N(\supp \spacecutoff u)$ and $\overline{N}(\supp\spacecutoff)\subseteq \overline{N}(\supp \spacecutoff u)$.

Hence,
\begin{equation}\label{N_eqn}
\overline N(\supp\spacecutoff)\subseteq WF_A(\spacecutoff u).
\end{equation}

Since the cutoff function $\spacecutoff$ is supported in a single coordinate patch, we can treat $\spacecutoff u$ and $\Op(\freqcutoff)\spacecutoff u$ as functions on $\R^n.$ Now, observe that $\widecheck{\freqcutoff} \ast\spacecutoff u = \Op(\freqcutoff) \spacecutoff  u$, where $\ast$ denotes standard convolution. This, along with \cite[Thm 8.6.15]{HormanderBook1983} gives
\begin{equation}\label{WF_eqn}
WF_A(\spacecutoff u) \subseteq WF_A(\Op(\freqcutoff) \spacecutoff u) \cup (\Rn \times \Gamma_A( \widecheck{\freqcutoff})^c).
\end{equation}
Applying \eqref{N_eqn}, we obtain
$$
\overline{N}(\supp \spacecutoff) \subseteq WF_A(\Op(\freqcutoff) \spacecutoff u)\cup (\Rn \times \Gamma_A( \widecheck{\freqcutoff})^c),
$$
and therefore,
\[\overline{N}(\supp \spacecutoff)  \cap (\Rn \times \Gamma_A(\widecheck\freqcutoff))\subseteq WF_A(\Op(\freqcutoff) \spacecutoff u).\]
By hypothesis, there exists a point
\[
(x_0,\xi_0)\in\overline N(\supp \spacecutoff)\cap (\Rn \times \Gamma_A(\widecheck\freqcutoff)) \subseteq WF_A(\Op(\freqcutoff)\spacecutoff u).\]
In particular, $x_0\in\supp\spacecutoff$, and since $\wt\spacecutoff\equiv 1$ on a neighborhood of $\supp\spacecutoff$, we see that $(x_0,\xi_0)$ must also lie inside $WF_A(\wt \spacecutoff\Op(\freqcutoff)\spacecutoff u) = WF_A(B u)$. This contradicts the assumption that $B u = 0$, and thus the proposition is proved.

\end{proof}

\begin{rmk}\textnormal{ It is worth noting that the argument of this lemma works for when $\Delta_g$ is replaced by $P$, an elliptic second order pseudodifferential operator, as long as $P$'s eigenfunctions do not vanish identically on open sets. 
}
\end{rmk}

\noindent Given \propref{unique_cont_condition}, the proof of \thmref{unique_cont_thm} is straightforward.
\begin{proof}[Proof of \thmref{unique_cont_thm}]
Given a real analytic manifold $(M,g),$ take $\spacecutoff,\wt \spacecutoff$ as in the statement of \propref{unique_cont_condition}. Let $(x_0,\xi_0) \in N_e(\supp \spacecutoff)$ be an arbitrary exterior normal. Then, take any $\freqcutoff\in C^\infty(\R^n)$ which is identically one in a conic neighborhood of $\xi_0$, zero on the complement of a slightly larger conic neighborhood, and homogeneous of degree 0 away from the origin. Then $\Gamma_A(\widecheck\freqcutoff )$ contains $\xi_0$ because $\freqcutoff \equiv 1$ on a conic neighborhood of $\xi_0$, and so one may take $\Phi \equiv 1$ in the definition of $\Gamma_A.$ \propref{unique_cont_condition} then guarantees that $B = \wt\spacecutoff\Op(\freqcutoff)\spacecutoff$ does not annihilate any eigenfunctions of $\Delta_g,$ and thus neither does $W = B^*B.$ One can repeat this process in any finite number of coordinate patches to show that there exists $W=\sum_{j=1}^N B_j^*B_j$ with the same property.
\end{proof}
We now construct a pseudodifferential damping coefficient on $\mathbb T^2$ which satisfies \assumsref{geometric_control}{unique_cont}.
\begin{example}
\normalfont
Let $\mathbb T^2 = \R^2/\Z^2$ denote the two-dimensional torus equipped with the flat metric, and let $\Delta$ be the associated Laplace-Beltrami operator. Let $\d>0$ and let $\spacecutoff_1\in C_c^\infty(\mathbb T^2)$ be supported in the vertical strip $\{(x^{(1)},x^{(2)})\in\mathbb T^2:\,\frac{1}{2}-\delta \le x^{(1)} \le \frac{1}{2} + \delta\}$ and equal to one on a smaller vertical strip. Define $\wt \spacecutoff_1$ in a similar way, but with $\wt \spacecutoff_1 \equiv 1$ on the support of $\spacecutoff_1$. Analogously, let $\spacecutoff_2\in C_c^\infty(\mathbb T^2)$ be supported in the horizontal strip $\{(x^{(1)},x^{(2)})\in\mathbb T^2:\,\frac{1}{2}-\delta \le x^{(2)} \le \frac{1}{2} + \delta\}$ and equal to one on a smaller horizontal strip, and define $\wt\spacecutoff_2$ similarly with $\wt\spacecutoff_2 \equiv 1$ on the support of $\spacecutoff_2$. 

Now, let $\ve>0$ and let $\freqcutoff_1\in C^\infty(\mathbb S^1)$ be supported in the set 
\[\Theta_{2\ve} = \lp-\frac{\pi}{4}-2\ve,\frac{\pi}{4}+2\ve\rp\cup\lp\frac{3\pi}{4}-2\ve,\frac{5\pi}{4}+2\ve\rp\]
and equal to one on the smaller set $\Theta_{\ve}$. Similarly, let $\freqcutoff_2\in C^\infty(\mathbb S^1)$ be nonzero on $\Theta_{2\ve}+\frac{\pi}{2}$ and equal to one on $\Theta_{\ve}+\frac{\pi}{2}$. Choose $\beta\in C_c^\infty(\R)$ to be supported in $[\frac{1}{4},\infty)$ and equal to one on $[\frac{1}{2},\infty)$. Then define symbols $b_j\in S_{cl}^0(T^*\mathbb T^2)$ by 
\[b_j(\xi) = \freqcutoff_j(\theta)\beta(r),\quad j = 1,2,\]
where $\xi = (r,\theta)$ in standard polar coordinates on $T_x^*\mathbb T^2.$ Figure \ref{torusimage} illustrates the cone of directions in $T_{x_0}^*\mathbb T^2$ in which $b_1$ is supported at some arbitrary $x_0\in\supp\chi_1.$ Now define $B_j = \wt\spacecutoff_j\textnormal{Op}(b_j)\spacecutoff_j$, and set the damping coefficient $W$ to be
\[W = B_1^*B_1 + B_2^*B_2.\]
To see $\ker W$ contains no nontrivial eigenfunctions of $\Delta$ we apply \propref{unique_cont_condition}. Note $\overline N\lp\supp \spacecutoff_1\rp$ contains all points of the form $(x,\xi)$ with $x\in\partial(\supp\spacecutoff_1)$ and $\xi = (r,\theta)$, where $\theta = 0$ or $\theta = \pi$. Since $b_1$ is constant in a conic neighborhood of both of these cotangent directions, the hypotheses of \propref{unique_cont_condition} are satisfied. Thus, $\ker B_1$ contains no eigenfunctions of the Laplacian. An analogous argument holds for $B_2$, and since $B_1^*B_1$ and $B_2^*B_2$ are nonnegative operators, $W$ cannot annihilate any eigenfunctions of the Laplacian.

\begin{figure}
\centering
\includegraphics[scale=0.45]{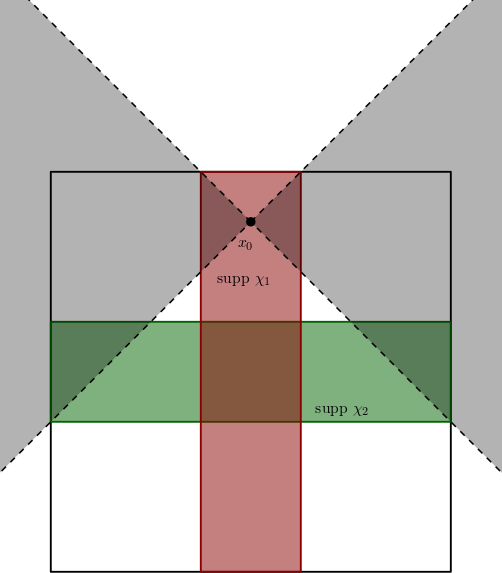}
\caption{The cone of directions containing the support of $b_1(x_0,\cdot)$ in $T_{x_0}^*\mathbb T^2$.}\label{torusimage}
\end{figure}

To show exponential energy decay with $W$ as the damping coefficient, we must also demonstrate that $W$ satisfies the AGCC. For this, it is convenient to observe that the AGCC is equivalent to the existence of some $T_0 > 0$ and $c > 0$ such that every trajectory $t\mapsto \varphi_t(x_0,\xi_0)$ encounters the set 
\[\mathscr W_c = \{(x,\xi)\in T^*\mathbb T^2:\, w(x,\xi) \ge c > 0\}\]
in time $T\le T_0.$ Recall that the geodesics on $\mathbb T^2$ are the projections of straight lines in $\R^2$ under the quotient map. Thus, the geodesic flow on $S^*\mathbb T^2$ is given by 
\[(x,\xi) \mapsto  ((x+t\xi)\textnormal{mod}\,\Z^2,\xi).\]
 Given an arbitrary point $(x_0,\xi_0)\in S^*\mathbb T^2$, we will show that $(\gamma(t),\gamma'(t)) = ((x_0+t\xi_0)\textnormal{mod}\,\Z^2,\xi_0)$ must intersect $\mathscr W_c$ in some fixed time $T_0>0$. Let us write $\xi_0\in \mathbb S^1$ as $(\cos\theta_0,\sin\theta_0)$, and consider the case where $\theta_0$ lies in $\Theta_{\ve}$. Suppose first that 
\[\theta_0 \in \lp-\frac{\pi}{4}-\ve,\frac{\pi}{4}+\ve\rp,\]
 which implies $b_1(\xi_0)\ne 0$. Then, if $x_0 = (x_0^{(1)},x_0^{(2)})$, the horizontal coordinate of $\gamma(t)$ is given by 
\[(x_0^{(1)} + t\cos\theta_0)\textnormal{mod}\,\Z,\]
 which must reach $\frac{1}{2}$ in some time less than $\frac{1}{\cos\theta_0} \le \frac{1}{\cos(\pi/4 + \ve)}$. Therefore, $(\gamma(t),\gamma'(t))$ intersects the region where $b_1$ is strictly positive in time less than $\frac{1}{\cos(\pi/4+\ve)}$. The same argument holds if instead ${\theta_0\in(\frac{3\pi}{4}-\ve,\frac{5\pi}{4} + \ve)}$, and so whenever $\theta_0\in \Theta_\ve$, we have that there exists a $c > 0$ such that $(\gamma(t),\gamma'(t))$ intersects $\{b_1(x,\xi) \ge \sqrt{c}\}$ in finite time. Analogously, if $\theta_0 \in\Theta_\ve + \frac{\pi}{2}$, then the vertical component of $\gamma(t)$, given by $(x_0^{(2)} + t\sin\theta_0)\!\!\mod\Z$, must equal $\frac{1}{2}$ in some time less than $\frac{1}{\sin(\pi/4 - \ve)}$. Therefore, $(\gamma(t),\gamma'(t))$ intersects $\{b_2(x,\xi)\ge \sqrt{c}\}$ in finite time. Since 
\[\mathbb T^2\times\lp\Theta_\ve\cup(\Theta_\ve + \frac{\pi}{2})\rp = S^*\mathbb T^2,\]
and since $w(x,\xi) = b_1^2(x,\xi) + b_2^2(x,\xi)$, we have that for every $(x_0,\xi_0)\in S^*\mathbb T^2$, the curve $\varphi_t(x_0,\xi_0)$ intersects $\mathscr W_{c}$ in some fixed time $T_0>0$. We have therefore shown that $W$ as defined here satisfies both Assumptions \ref{geometric_control} and \ref{unique_cont}. Thus by \thmref{exp_decay_thm}, all solutions to the damped wave equation on $\mathbb T^2$ with damping coefficient $W$ experience exponential energy decay. 

\end{example}

\begin{rmk}
\normalfont
In the previous example, one may notice that on the intersection of the vertical and horizontal strips, the principal symbol of the damping coefficient is supported in all directions $\xi\in T^*\mathbb T^2\setminus 0.$ So in this region, $W$ behaves very much like a multiplication operator for frequencies away from zero. A natural question is whether or not there must always be a point of ``full microsupport" if the hypotheses of \thmref{exp_decay_thm} are to be satisfied. In fact, there need not be such a point. To see this, we can modify our example above as follows.

Define $\spacecutoff_1,\,\wt\spacecutoff_1,\,\spacecutoff_2,\,\wt\spacecutoff_2$ and $b_1$ in a similar fashion to the previous example, but now define $b_2$ to be supported only in the directions with angle $\theta\in (\frac{\pi}{4}-2\ve,\frac{3\pi}{4} + 2\ve)$ and identically one on $(\frac{\pi}{4}-\ve,\frac{3\pi}{4} + \ve)$. Next, we introduce another horizontal strip, disjoint from the first, with a corresponding pair of cutoff functions $\spacecutoff_3,\wt\spacecutoff_3$. Then, define $\freqcutoff_3\in C^\infty(\mathbb S^1)$ to be supported in $(\frac{5\pi}{4} - 2\ve,\frac{7\pi}{4} + 2\ve)$ and equal to one on $(\frac{5\pi}{4} - \ve,\frac{7\pi}{4} + \ve)$, and let $b_3(\xi) = \freqcutoff_3(\theta)\beta(r)$, where $\xi = (r,\theta)$ as before. This is illustrated in Figure \ref{torusimage3strips}. Then, if we define $B_3 = \wt\spacecutoff_3 \Op(b_3)\spacecutoff_3$ and set $W = \sum_{j=1}^3 B_j^*B_j,$ we can apply arguments similar to those above to see that Assumptions \ref{geometric_control} and \ref{unique_cont} are still satisfied, but there does not exist any point $x\in\mathbb T^2$ where $w(x,\xi)$ is supported in all directions.
\end{rmk}
\begin{figure}[H]
\centering
\includegraphics[scale=0.45,trim = 0 0.2cm 0 0.2cm]{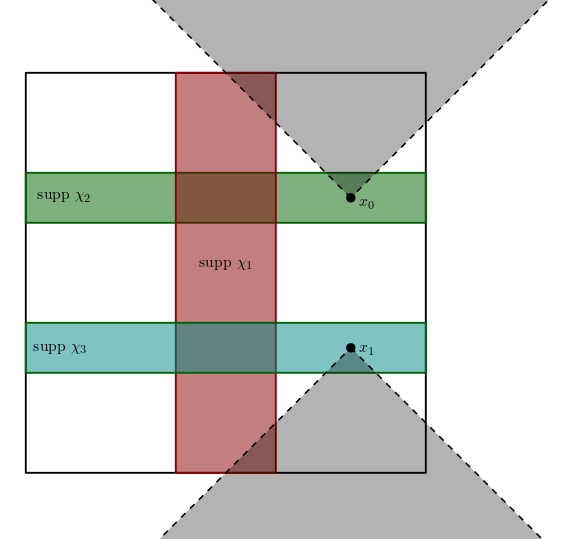}
\caption{The cones containing the supports of  $b_2(x_0,\cdot)$ and $b_3(x_1,\cdot)$.}\label{torusimage3strips}
\end{figure}

\bibliography{master_bib}{}
\bibliographystyle{halpha-abbrv}

\end{document}